%% file: ms.tex
\documentclass[paper=a4, fontsize=11pt]{amsart}
\input{preamble.tex}
\input{definitions.tex}

\title{Conormal Varieties on the Cominuscule Grassmannian -- II}
\author{Rahul Singh} 

\begin{document}
\input{abstract.tex}
\maketitle 


\input{intro.tex}
\setcounter{mainTheorem}{0}
\input{basics.tex}

\input{resolution.tex}

\input{sln.tex}

\input{typeC.tex}

\input{extension.tex} 
\input{suff.tex}
\input{orbital.tex}
\input{remarks.tex}

\bibliography{biblio}
\bibliographystyle{amsalpha}

\noindent
 {\scshape Rahul Singh}\\
 {\scshape Department of Mathematics, Northeastern University,\\ Boston, MA 02115, USA}.\\
 {\itshape E-mail address}: \texttt{singh.rah@husky.neu.edu}
\end{document}

%% file: preamble.tex
\usepackage{amsmath,amsthm,amsfonts,amssymb}
\usepackage{framed}
\usepackage{nicefrac}
\usepackage[english]{babel}
\usepackage{color,xcolor,graphicx}
\usepackage{tikz}
\usetikzlibrary{cd}
\usepackage[capitalize]{cleveref}
\theoremstyle{plain}
\newtheorem{lemma}[subsection]{Lemma}

\newtheorem*{theorem*}{Theorem}
\newtheorem{proposition}[subsection]{Proposition}
\newtheorem{corollary}[subsection]{Corollary}

\newtheorem{mainTheorem}{Theorem}

\crefname{mainTheorem}{Theorem}{Theorems}
\crefname{equation}{Equation}{Equations}

\theoremstyle{definition}

\theoremstyle{remark}
\newtheorem{example}[subsection]{Example}

\newtheorem*{remark*}{Remark}
\newtheorem*{caution}{Caution}

\numberwithin{equation}{section} 
\numberwithin{figure}{section} 
\numberwithin{table}{section} 
\numberwithin{subsection}{section} 

\makeatletter
\let\c@equation\c@subsection
\let\c@figure\c@subsection
\let\c@table\c@subsection
\makeatother

%% file: definitions.tex
\def\mid{\,\middle\vert\,}
\newcommand\set[2]{\ensuremath{\left\{#1\mid#2\right\}}}
\newcommand\up[2]{\ensuremath{#1_{#2}}}
\def\u{\ensuremath{\mathfrak u} }
\def\lgl{\ensuremath{\mathfrak g} }
\def\con{\ensuremath{T_X^*X_w}}
\DeclareMathOperator{\im}{Im}
\def\F#1.{\ensuremath{F(#1)}}
\def\E#1.{\ensuremath{E(#1)}}
\newcommand{\Ni}[1][ ]{\ensuremath{{\mathcal N_{\boldsymbol{#1}}} } }
\def\W{\ensuremath{S_n}}

\def\mod{\ensuremath{\mathrm{mod}\,}}
\def\define{\ensuremath{\overset{\scriptscriptstyle\operatorname{def}}=}}
\def\roots#1.#2.{\ensuremath{\Delta^{#1}_{#2}}}
\def\k{\ensuremath{\mathsf k}}

\DeclareMathOperator{\rk}{rk}
\def\mw{\mathrm m_w}

\def\w{{\underline w}}
\def\reso{\ensuremath{\widetilde{Z_{\w}}} }
\def\resO{(\reso)^\circ}
\def\bsw{\ensuremath{\widetilde{X_{\w}}} }
\def\dynk{\ensuremath{\mathcal D}}
\def\ws{\ensuremath{\omega}}

%% file: abstract.tex
\begin{abstract}
Let $X_w$ be a Schubert subvariety of a cominuscule Grassmannian $X$, and let $\mu:T^*X\rightarrow\mathcal N$ be the Springer map from the cotangent bundle of $X$ to the nilpotent cone $\mathcal N$. 
In this paper, we construct a resolution of singularities for the conormal variety $T^*_XX_w$ of $X_w$ in $X$.
Further, for $X$ the usual or symplectic Grassmannian, we compute a system of equations defining $T^*_XX_w$ as a subvariety of the cotangent bundle $T^*X$ set-theoretically.
This also yields a system of defining equations for the corresponding orbital varieties $\mu(T^*_XX_w)$.
Inspired by the system of defining equations, we conjecture a type-independent equality, namely $T^*_XX_w=\pi^{-1}(X_w)\cap\mu^{-1}(\mu(T^*_XX_w))$.
The set-theoretic version of this conjecture follows from this work and previous work for any cominuscule Grassmannian of type A, B, or C.
\end{abstract}

%% file: intro.tex
We work over an \emph{algebraically closed field} \k\ of \emph{good characteristic} (for a definition, see \cite{MR794307}).
Let $G$ be a connected algebraic group whose Lie algebra \lgl is simple.

For $\mathcal P$ a conjugacy class of parabolic subgroups of $G$, we denote by $X^{\mathcal P}$ the variety of parabolic subgroups of $G$ whose conjugacy class is $\mathcal P$.
The cotangent bundle $T^*X^{\mathcal P}$ of $X^{\mathcal P}$ is given by\[
	T^*X^{\mathcal P}=\set{(P,x)\in X^{\mathcal P}\times\Ni}{x\in\u_P},
\]
where \Ni is the variety of nilpotent elements in \lgl.
The map $\mu:T^*X\rightarrow\Ni$, given by $\mu(P,x)=x$, is the celebrated \emph{Springer map}.

Let $\mathcal B$ be the conjugacy class of Borel subgroups of $G$.
The Steinberg variety, 
\[
	\mathcal Z^{\mathcal P}=\set{(B,P,x)\in X^{\mathcal B}\times X^{\mathcal P}\times\Ni}{x\in\u_B\cap\u_P},
\] 
is reducible.
Each irreducible component $\mathcal Z^{\mathcal P}_w$ of $\mathcal Z^{\mathcal P}$ is the conormal variety of a $G$-orbit closure (under the diagonal action) in $X^{\mathcal B}\times X^{\mathcal P}$.

In this paper, for certain choices of $\mathcal P$, namely the ones for which $X^{\mathcal P}$ is cominuscule (see \cref{sub:comin}), we construct a resolution of singularities for each irreducible component $\mathcal Z_w^{\mathcal P}\subset\mathcal Z^{\mathcal P}$.
In types A and C, we also provide a system of defining equations, for each component $\mathcal Z^{\mathcal P}_w$ as a subvariety in $X^{\mathcal B}\times X^{\mathcal P}\times\Ni$.
This also yields a system of defining equations for certain orbital varieties. 
We discuss this later in this section.

\subsection*{} \hspace{2pt} 
Before getting into the details, let us first present the irreducible components of $\mathcal Z^{\mathcal P}$ from an alternate point of view.
We fix a Borel subgroup $B$ in $G$, and a standard parabolic subgroup $P$ corresponding to omitting a cominuscule simple root $\gamma$, see \cref{sub:comin}.
Let $X$ be the variety of `parabolic subgroups conjugate to $P$'.  
We have an isomorphism $X\cong\nicefrac GP$, and further, a $G$-equivariant isomorphism,\[
	G\times^B X\xrightarrow\sim X^{\mathcal B}\times X,
\]
given by $(g,P)\mapsto (gBg^{-1},gPg^{-1})$.

A $B$-orbit $C_w\subset X$ is called a Schubert cell, and its closure $X_w$ is called a Schubert variety.
The conormal variety $T^*_XX_w$ of $X_w$ in $X$ is simply the closure of the conormal bundle of $C_w$ in $X$, see \cref{sub:con}.

Consider the map $\mathcal Z_w^{\mathcal P}\rightarrow X^{\mathcal B}$, given by $(B,P,x)\mapsto B$.
We identify $\mathcal Z_w^{\mathcal P}$ as a fibre bundle over $X^{\mathcal B}$ via this map, with the fibre over the point $B\in X^{\mathcal B}$ being precisely the conormal variety $T^*_XX_w$.
In particular, we have an isomorphism $$G\times^B T^*_XX_w\xrightarrow\sim\mathcal Z_w^{\mathcal P}.$$ 
From this viewpoint, it is clear that the geometry of $T^*_XX_w$ is closely related to the geometry of $\mathcal Z^{\mathcal P}_w$.

\begin{remark*}
A similar (and essentially equivalent) statement can be found in \cite[Proposition 3.3.4]{MR1433132}; the proof there is different, leveraging the fact that the Springer map $\mu:T^*X\rightarrow\Ni$ can be identified with the moment map arising from the $G$-symplectic structure on $T^*X$.
\end{remark*}

\subsection*{}\hspace{2pt}
We now present our main results.
Let $X_w$ be a Schubert subvariety of a cominuscule Grassmannian $X$. 
In \cref{section2}, we present a variety $\reso$, which is a vector bundle over a Bott-Samelson variety resolving $X_w$, along with a proper birational $B$-equivariant map, $\theta_\w:\reso\rightarrow\con$.

\begin{mainTheorem}
The map $\theta_\w:\reso\rightarrow\con $ is a $B$-equivariant resolution of singularities.
\end{mainTheorem}
Since the map $\theta_\w$ is $B$-equivariant, it also yields a resolution of singularities,\[
	G\times^B\theta_\w:G\times^B\reso\rightarrow\mathcal Z^{\mathcal P}_w.
\]
of the Steinberg component $\mathcal Z^{\mathcal P}_w$.

\subsection*{} \hspace{2pt} 
Next, we study the system of defining equations for the conormal variety $T^*_XX_w$ inside $T^*X$. 
For $i\geq 1$, let \E i. denote a vector space with basis $e_1,\cdots e_i$.
We fix a non-degenerate skew-symmetric bilinear form \ws\ on \E 2d.. 
Let $G$ either $SL(\E n.)$ or $Sp(\E{2d}.,\ws)$, and accordingly, let $X$ be either the \emph{usual Grassmannian},\[
	Gr(d,n)	=\left\{V\subset\E n.\mid\dim V=d\right\},
\]
or the \emph{symplectic Grassmannian},\[
	SGr(2d)	=\left\{V\subset\E 2d.\mid V=V^\perp\right\}.
\]
The cotangent bundle of $X$ is given by\[
	T^*X=\set{(V,x)\in X\times\Ni}{\im x\subset V\subset\ker x},
\]
where, recall that $\mathcal N$ denotes the corresponding nilpotent cone.
Let $B$ be the Borel subgroup which is the stabilizer of the flag $(\E i.)_i$ in $G$.
In \cref{finalEquations}, we provide a system of defining equations for $T^*_XX_w$ in $T^*X$.

\begin{mainTheorem}
Consider $(V,x)\in T^*X$.
Then $(V,x)\in T_X^*X_w$ if and only if $V\in X_w$, and further, for all $1\leq j<i\leq l+1$, we have
\begin{align*}
    \dim(\nicefrac{x\E\up ti.}{\E t_j.})\leq\begin{cases}\up r{i-1}-\up rj,\\ \up ci-\up c{j+1}.\end{cases}
\end{align*}
The numbers $r_i,c_i,t_i$ are defined in terms of $w$, see \cref{rici}.
\end{mainTheorem}
A system of defining equations for $\mathcal Z_w^{\mathcal P}$ in $X^{\mathcal B}\times X\times\Ni$ follows as a consequence.
We simply replace the subspaces $\E t_i.$ with the subspace $E'({t_i})$, 
where $(E'(i))_i$ 
is the flag fixed by $B'$, the Borel subgroup at the first coordinate in $X^{\mathcal B}\times X\times\Ni$.

\Cref{finalEquations} does not hold for the orthogonal Grassmannian.
The key difference between $\mathsf C_n$ and $\mathsf D_n$ is the following:
Consider the embedding of the Weyl group $W$ into $S_{2n}$.
Then, for $\mathsf C_n$, the Bruhat order on $W$ is identical to the order induced by restricting the (type A) Bruhat order on $S_{2n}$. 
This is not true for $\mathsf D_n$.

In \cref{genEqn}, we interpret \cref{finalEquations} in a type-independent manner, 
\begin{align*}
	\con =\mu^{-1}(\mu(\con ))\cap\pi^{-1}(X_w).
\end{align*}
Here $\mu:T^*X\rightarrow\Ni$ is the Springer map, and $\pi:T^*X\rightarrow X$ is the structure map defining the cotangent bundle. 

We conjecture that \Cref{genEqn} holds for any Schubert variety $X_w$ in any cominuscule Grassmannian $X$.
The containment $\subset$ holds trivially.
Besides \cref{finalEquations}, further evidence in support of this conjecture is provided by \cref{whenSmooth}, which states that \cref{genEqn} holds set-theoretically if $X_w$ is smooth, and by \cref{oppSmooth}, which states that \cref{genEqn} holds scheme-theoretically if the opposite Schubert variety $X_{w_0w}$ is smooth.
Combining these results, we see that \cref{genEqn} holds set-theoretically for any cominuscule Grassmannian in types A, B, and C.

\subsection*{} \hspace{2pt} 
Finally, let us discuss orbital varieties, and their relationship with the conormal varieties of Schubert varieties.
Consider a $G$-orbit $\mathcal N_{\boldsymbol\lambda}^\circ\subset\Ni$.
The irreducible components of the closure $\overline{\Ni[\lambda]\cap\u_B}$ are called orbital varieties.
The reader might consult \cite{MR2510045} for a general survey.

\begin{caution}
Some authors define an orbital variety to be an irreducible component of $\mathcal N_{\boldsymbol\lambda}^\circ\cap\u_B$, where $\mathcal N_{\boldsymbol\lambda}^\circ$ is $G$-orbit in $\Ni$.
What we call an orbital variety here is an orbital variety closure in their language.
\end{caution}

The key fact relating orbital varieties with conormal varieties is the following:
Given a conjugacy class $\mathcal P$, and a Schubert variety $X_w^{\mathcal P}$, the image of the conormal variety $T^*_{X^{\mathcal P}}X^{\mathcal P}_w$ under the Springer map $\mu$ is an orbital variety.
Conversely, every orbital variety is of the form $\mu(T^*_{X^{\mathcal B}}X^{\mathcal B}_w)$ for some Schubert variety $X^{\mathcal B}_w\subset X^{\mathcal B}$.

\Cref{finalEquations} yields equations for the corresponding orbital varieties.

\begin{mainTheorem}
Let $G$, $B$, $P$, $X$, $w$, and $\mu$ be as in \cref{finalEquations}.
Then\[
	\mu(T^*_XX_w)=\set{x\in\u_B}{x^2=0,\,\dim(\nicefrac{x\E t_i.}{\E t_j.})\leq\begin{cases}\up r{i-1}-\up rj,\\\up ci-\up c{j+1},\end{cases}\forall\,1\leq i<j\leq l}.
\]
\end{mainTheorem}

It is in general an open problem to give a combinatorial description of the inclusion order on orbital varieties.
For varieties of matrices $x$ satisfying $x^2=0$, this problem was solved by Melnikov \cite{MR2134184} in type A, and in types B and C by Melnikov and Barnea \cite{MR3717217}.
%
In \cref{recover}, we show how their results can be recovered as a simple consequence of \cref{mainOrbital}.

\subsection*{Acknowledgements} 
We thank Anna Melnikov for pointing out a serious error in an earlier version of this article. 
We thank Manoj Kummini, V. Lakshmibai, Anna Melnikov, and Dinakar Muthiah for some very illuminating conversations.

%% file: basics.tex
\section{The Conormal Variety of a Schubert Variety}\label{sec:basics}
In this section, we recall some standard results about Schubert varieties, their conormal varieties, and cominuscule Grassmannians.

Let \k\ be an \emph{algebraically closed field} of \emph{good characteristic}, \lgl a simple Lie algebra over \k, and $G$ a connected algebraic group for which $\lgl=\mathrm{Lie}(G)$.
We fix a maximal torus $T$ in $G$, and a Borel subgroup $B$ containing $T$.

Let \roots.. be the root system of \lgl with respect to $\mathfrak t=\mathrm{Lie}(T)$, and let $S$ and \roots+.. be the set of simple roots and positive roots respectively, corresponding to the choice of Borel subalgebra $\mathfrak b=\mathrm{Lie}(B)$.
For $\alpha\in\roots..$, we will write $\lgl_\alpha$ for the corresponding root space.

\subsection{Standard Parabolic Subgroups}
A subgroup $Q\subset G$ is called parabolic if the quotient $X^Q\define\nicefrac GQ$ is proper.
We will say that $Q$ is a \emph{standard} parabolic subgroup if $B\subset Q$.

Let $\set{s_\alpha}{\alpha\in S}$ be the set of simple reflections in the \emph{Weyl group} $W=\nicefrac{N_G(T)}T$; here $N_G(T)$ is the normalizer of $T$ in $G$. 
For any subset $R\subset S$, we have a subgroup $W_R\subset W$, given by $W_R=\left\langle s_\alpha\mid\alpha\in R\right\rangle$.
The subgroup $BW_RB\subset G$, given by,\[
	BW_RB=\set{b_1wb_2}{b_1,b_2\in B,\,w\in W_R},
\]
is a standard parabolic subgroup; further, the map $R\leftrightarrow BW_RB$ is a bijective correspondence from subsets of $S$ to the standard parabolic subgroups of $G$.

\subsection{Schubert Varieties}
Let $Q$ be a standard parabolic subgroup of $G$, corresponding to some subset $S_Q\subset S$.
A $B$-orbit $C^Q_w\subset X^Q$ is called a \emph{Schubert cell}.
The pull-back of $C^Q_w$ along the quotient map $G\rightarrow\nicefrac GQ=X^Q$ is\[
	BwQ=\set{bwq}{b\in B,\,q\in Q}.
\]

The closure $X^Q_w$ of the Schubert cell $C^Q_w$ is called a \emph{Schubert variety}. 
The Schubert varieties in $X_w^Q\subset X^Q$ are indexed by $w\in W^Q$, where,\begin{align}\label{minW}
    W^Q\define\left\{w\in W\mid w(\alpha)>0,\forall\alpha\in S_Q\right\}.
\end{align}
The set $W^Q$ is called the set of minimal representatives of $W$ with respect to $Q$.

\subsection{Bott-Samelson Varieties}
\label{bsw}
Let $\underline w=(s_1,\cdots,s_r)$ be a minimal word for $w$, i.e., the $s_i$ are simple reflections such that $w=s_1\cdots s_r$, and further, there is no sub-sequence of $\w$ whose product is $w$.

Let $P_i$ be the standard parabolic subgroup $\overline{Bs_iB}$. 
The Bott-Samelson variety,\[
    \bsw\define\nicefrac{P_1\times^B\cdots\times^B P_r}B,
\]
provides a resolution of singularities of $X_w^Q$ via the map $\rho^Q_\w:\bsw\rightarrow X^Q_w$, given by,\[
    (p_1,\cdots,p_r)\mapsto p_1\cdots p_r(\mod Q).
\]
Let $P_i^\circ$ denote the open set $Bs_iB\subset P_i$. 
The map $\rho^Q_\w$ induces an isomorphism,\[
	(\bsw)^\circ=\nicefrac{P_1^\circ\times^B\cdots\times^BP_r^\circ}B\xrightarrow\sim C_w^Q.
\]

\subsection{The Cotangent Bundle $T^*X^Q$}
The cotangent bundle $\pi:T^*X^Q\rightarrow X^Q$ is the vector bundle whose fibre $T^*_pX^Q$ at any point $p\in X^Q$ is precisely the cotangent space of $X^Q$ at $p$. 
We call $\pi$ the \emph{structure map} defining the cotangent bundle.

Recall that the characteristic of \k\ is a good prime.
We have (cf. \cite[Ch. 5]{MR2107324}),
\begin{align}\label{defn:cotan}
	T^*X^Q=G\times^Q\u_Q&=\nicefrac{(G\times\u_Q)}Q,
\end{align}
where the quotient is with respect to the action $q\boldsymbol\cdot(g,x)=(gq,Ad(q^{-1}x))$. 

\subsection{The Springer Map}\label{sub:springer}
Let \Ni be the nilpotent cone of $\mathfrak g$, i.e.,\[
    \Ni=\left\{x\in\lgl\mid Ad(x)\text{ is nilpotent}\right\}.
\]
The Springer map $\mu^Q:T^*X^Q\rightarrow\Ni$, given by, \[
	\mu^Q(g,x)=Ad(g)x,
\]
is a proper map. 
The product map, \begin{align*}
	(\pi,\mu^Q):T^*X^Q\rightarrow X^Q\times\Ni,	&&(g,x)\mapsto (g,\mu^Q(x)), 
\end{align*}
is a closed immersion, see \cite{MR1433132} for details.

\subsection{The Conormal Variety of a Schubert Variety}\label{sub:con}
Let $C^Q_w$ (resp. $X^Q_w$) be a Schubert cell (resp. Schubert variety) in $X^Q$, corresponding to some $w\in W^Q$.
The conormal bundle of $C^Q_w$ in $X^Q$ is the vector bundle,\[
	\pi^\circ_w:T^*_{X^Q}C^Q_w\rightarrow C^Q_w,
\]
whose fibre at a point $p\in C^Q_w$ is precisely the annihilator of the tangent subspace $T_pC^Q_w$ in $T^*_pX^Q $, i.e.,\[
	(T^*_{X^Q}C_w^Q)_p=\set{x\in T^*_pX^Q}{x(v)=0,\forall v\in T_pC^Q_w}.
\]

The \emph{conormal variety} $T^*_{X^Q}X^Q_w$ of $X^Q_w$ in $X^Q$ is the closure (in $T^*X^Q$) of the conormal bundle $T^*_{X^Q}C_w^Q$. 
The restriction of the structure map $\pi:T^*X^Q\rightarrow X^Q$ to the conormal variety induces a \emph{structure map}, $\pi_w:T^*_{X^Q}X_w^Q\rightarrow X^Q_w$.

\subsection{Cominuscule Grassmannians}\label{sub:comin}
A simple root $\gamma\in S$ is called \emph{cominuscule} if the coefficient of $\gamma$ in any positive root is either $0$ or $1$, i.e.,\[
	\alpha\in\roots+..\implies 2\gamma\not\leq\alpha. 
\]
The cominuscule roots for various Dynkin diagrams are labelled in \cref{TBL:comin}. 

\input{cotcomin_table-1.tex}

\begin{example}\label{uGrass}
Let \E n. be an $n$--dimensional vector space.
The variety of $d$--dimensional subspaces of \E n. is called the usual Grassmannian variety,\[
    Gr(d,n)\define\left\{V\subset\E n.\mid\dim V=d\right\}.
\] 
It is a cominuscule Grassmannian corresponding to the group $G=SL_n$, and the cominuscule root $\alpha_d\in\mathsf A_{n-1}$, see \cref{TBL:comin}.
\end{example}

\begin{example}\label{sGrass}
Let \E 2d. be a $2d$--dimensional vector space, and $\omega$ a symplectic form on $\E 2d.$.
The variety of Lagrangian subspaces in \E 2d.,\[
    SGr(2d)\define\left\{V\subset\E 2d.\mid V=V^\perp\right\},
\] 
is called the \emph{symplectic Grassmannian}.
It is a cominuscule Grassmannian corresponding to the group $G=Sp_{2d}$, and the cominuscule root $\alpha_d\in\mathsf C_d$, see \cref{TBL:comin}.
\end{example}

%% file: cotcomin_table-1.tex
\begin{table}

  \begin{tikzpicture}[scale=.4]
    \draw (-3.5,0) node[anchor=east]  {$\mathsf A_{n}$};
    \foreach \x in {-2,0,2,4,6}
    \draw[thick,fill=black!70] (\x cm,0) circle (.3cm);
    \draw[dotted, thick] (-1.7 cm,0) -- +(1.4 cm,0);
    \foreach \y in {0.3,2.3}
    \draw[thick] (\y cm,0) -- +(1.4 cm,0);
    \draw[dotted, thick] (4.3 cm,0) -- +(1.4 cm,0);
    \draw (-2,.8) node {$\scriptscriptstyle{1}$};
    \draw (0,.8) node {$\scriptscriptstyle{d-1}$};
    \draw (2,.8) node {$\scriptscriptstyle{d}$};
    \draw (4,.8) node {$\scriptscriptstyle{d+1}$};
    \draw (6,.8) node {$\scriptscriptstyle{n}$};
    
    \draw (11,0) node[anchor=east]  {$\mathsf B_{n}$};
    \foreach \x in {14.5,16.5,18.5,20.5}
    \draw[thick,fill=white!70] (\x cm,0) circle (.3cm);
    \draw[thick,fill=black!70] (12.5 cm,0) circle (.3cm);
    \draw[thick] (12.8 cm,0) -- +(1.4 cm,0);
    \draw[dotted,thick] (14.8 cm,0) -- +(1.4 cm,0);
    \draw[thick] (16.8 cm,0) -- +(1.4 cm,0);
    \draw[thick] (18.8 cm, .1 cm) -- +(1.4 cm,0);
    \draw[thick] (18.8 cm, -.1 cm) -- +(1.4 cm,0);
    \draw[thick] (19.4 cm, .3 cm) -- +(.3 cm, -.3 cm);
    \draw[thick] (19.4 cm, -.3 cm) -- +(.3 cm, .3 cm);
    \draw (12.5,.8) node {$\scriptscriptstyle{1}$};
    \draw (14.5,.8) node {$\scriptscriptstyle{2}$};
    \draw (16.5,.8) node {$\scriptscriptstyle{n-2}$};
    \draw (18.5,.8) node {$\scriptscriptstyle{n-1}$};
    \draw (20.5,.8) node {$\scriptscriptstyle{n}$};
    
    \draw (22,0) node[anchor=east]{};
    
  \end{tikzpicture}
  
    \vspace{3mm}

  \begin{tikzpicture}[scale=.4]
    \draw (-0.8,0) node[anchor=east]  {$\mathsf C_{n}$};
    \foreach \x in {0.5,2.5,4.5,6.5}
    \draw[thick,fill=white!70] (\x cm,0) circle (.3cm);
    \draw[thick,fill=black!70] (8.5 cm,0) circle (.3cm);
    \draw[thick] (0.8 cm,0) -- +(1.4 cm,0);
    \draw[dotted,thick] (2.8 cm,0) -- +(1.4 cm,0);
    \draw[thick] (4.8 cm,0) -- +(1.4 cm,0);
    \draw[thick] (6.8 cm, .1 cm) -- +(1.4 cm,0);
    \draw[thick] (6.8 cm, -.1 cm) -- +(1.4 cm,0);
    \draw[thick] (7.4 cm, 0 cm) -- +(.3 cm, .3 cm);
    \draw[thick] (7.4 cm, 0 cm) -- +(.3 cm, -.3 cm);
    \draw (0.5,.8) node {$\scriptscriptstyle{1}$};
    \draw (2.5,.8) node {$\scriptscriptstyle{2}$};
    \draw (4.5,.8) node {$\scriptscriptstyle{n-2}$};
    \draw (6.5,.8) node {$\scriptscriptstyle{n-1}$};
    \draw (8.5,.8) node {$\scriptscriptstyle{n}$};
    
    \draw (14,0) node[anchor=east]  {$\mathsf D_{n}$};
    \foreach \x in {17,19,21,23}
    \draw[thick,fill=white!70] (\x cm,0) circle (.3cm);
    \draw[thick,fill=black!70] (15 cm,0) circle (.3cm);
    \draw[xshift=23 cm,thick,fill=black!70] (30: 17 mm) circle (.3cm);
    \draw[xshift=23 cm,thick,fill=black!70] (-30: 17 mm) circle (.3cm);
    \draw[dotted,thick] (19.3 cm,0) -- +(1.4 cm,0);
    \foreach \y in {15.3, 17.3,21.3}
    \draw[thick] (\y cm,0) -- +(1.4 cm,0);
    \draw[xshift=23 cm,thick] (30: 3 mm) -- (30: 14 mm);
    \draw[xshift=23 cm,thick] (-30: 3 mm) -- (-30: 14 mm);
    \draw (15,.8) node {$\scriptscriptstyle{1}$};
    \draw (17,.8) node {$\scriptscriptstyle{2}$};
    \draw (19,.8) node {$\scriptscriptstyle{3}$};
    \draw (21,.8) node {$\scriptscriptstyle{n-3}$};
    \draw (23,.8) node {$\scriptscriptstyle{n-2}$};
    \draw (24.45,1.6) node {$\scriptscriptstyle{n-1}$};
    \draw (24.45,-.2) node {$\scriptscriptstyle{n}$};
  \end{tikzpicture}

    \vspace{1mm}

  \begin{tikzpicture}[scale=.4]
    \draw (-2.4,0) node[anchor=east]  {$\mathsf E_{6}$};
    \foreach \x in {1,3,5}
    \draw[thick,fill=white!70] (\x cm,0) circle (.3cm);
    \draw[thick,fill=white!70] (3 cm, 2 cm) circle (.3cm);
    \foreach \x in {-1,7}
    \draw[thick,fill=black!70] (\x cm,0) circle (.3cm);
    \foreach \y in {-0.7, 1.3, 3.3, 5.3}
    \draw[thick] (\y cm,0) -- +(1.4 cm,0);
    \draw[thick] (3 cm,.3 cm) -- +(0,1.4 cm);
    \draw (-1,.8) node {$\scriptscriptstyle{1}$};
    \draw (1,.8) node {$\scriptscriptstyle{3}$};
    \draw (3.5,.8) node {$\scriptscriptstyle{4}$};
    \draw (5,.8) node {$\scriptscriptstyle{5}$};
    \draw (7,.8) node {$\scriptscriptstyle{6}$};
    \draw (3.5,2.8) node {$\scriptscriptstyle{2}$};

    \draw (12,0) node[anchor=east]  {$\mathsf E_{7}$};
    \foreach \x in {13,15,17,19,21}
    \draw[thick,fill=white!70] (\x cm,0) circle (.3cm);
    \draw[thick,fill=white!70] (19 cm, 2 cm) circle (.3cm);
    \draw[thick,fill=black!70] (23 cm,0) circle (.3cm);
    \foreach \y in {13.3, 15.3, 17.3, 19.3, 21.3}
    \draw[thick] (\y cm,0) -- +(1.4 cm,0);
    \draw[thick] (19 cm,.3 cm) -- +(0,1.4 cm);
    \draw (13,.8) node {$\scriptscriptstyle{1}$};
    \draw (15,.8) node {$\scriptscriptstyle{3}$};
    \draw (17,.8) node {$\scriptscriptstyle{4}$};
    \draw (19.5,.8) node {$\scriptscriptstyle{5}$};
    \draw (21,.8) node {$\scriptscriptstyle{6}$};
    \draw (23,.8) node {$\scriptscriptstyle{7}$};
    \draw (19.5,2.8) node {$\scriptscriptstyle{2}$};
  \end{tikzpicture}
  
 \vspace{3mm} 
    \caption{Dynkin diagrams with cominuscule simple roots marked in black.}
    \label{TBL:comin}
\end{table}

%% file: resolution.tex
\section{A Resolution of Singularities of the Conormal Variety}
\label{section2}
Let $G$, $B$, $T$, $\roots..$, $\roots+..$, and $S$ be as in the previous section.
We fix a cominuscule root $\gamma\in S$.
Let $P$ be the standard parabolic subgroup corresponding to $S\backslash\{\gamma\}$, and let \u be the Lie algebra of the unipotent radical of $P$.
We will denote the variety $X^P=\nicefrac GP$ as simply $X$, and the Schubert varieties $X_w^P$ as simply $X_w$.

In this section, we study the conormal variety $T_X^*X_w$ of a Schubert variety $X_w$ in $X$.
In particular, we describe the structure of the \emph{conormal bundle} $T^*_XC_w$ in \cref{lem:zow}, and construct a resolution of singularities of $T_X^*X_w$ in \cref{bir0}.

\begin{lemma}\label{defn:uw}
For any $w\in W^P$, the subspace $\u_w\define\u\cap Ad(w^{-1})\u_B$  is $B$-stable.
\end{lemma}
\begin{proof}
The subspaces $\u$ and $Ad(w^{-1})\u_B$ are $T$-stable, and so their intersection $\u_w$ is also $T$-stable.
Further, since $Ad(w^{-1})\lgl_\alpha=\lgl_{w^{-1}(\alpha)}$, we have,\[
	\u_w=\bigoplus\limits_{\alpha\geq\gamma}\lgl_\alpha\bigcap\bigoplus\limits_{\alpha\in\roots+..}\lgl_{w^{-1}(\alpha)}=\bigoplus\limits_{\alpha\in R}\lgl_\alpha,
\]
where $R=\left\{\alpha\in\roots..\mid\alpha\geq\gamma,\,w(\alpha)>0\right\}$.
Since $B$ is generated by the torus $T$ and the root subgroups $U_\alpha$, $\alpha\in S$,
it suffices to show that $\u_w$ is $U_\alpha$-stable for all $\alpha\in S$.
This follows from the claim,\[
	\alpha\in R,\beta\in S, \alpha+\beta\in\roots..\implies\alpha+\beta\in R,
\]
which we now prove.
We first consider the case $\beta=\gamma$.
For any $\alpha\in R$, we have $\gamma\leq\alpha$, hence $2\gamma\leq\alpha+\beta$.
Now, since $\gamma$ is cominuscule, we have $\alpha+\beta\not\in\roots..$.

Next, we consider $\beta\in S\backslash\{\gamma\}$.
In this case, since $w\in W^P$, it follows from \cref{minW} that $w(\beta)>0$. 
Now, for any $\alpha\in R$, we have $w(\alpha)>0$, hence\[
    w(\alpha+\beta)=w(\alpha)+w(\beta)>0.
\]
It follows from the definition of $R$ that if $\alpha+\beta\in\roots..$, then $\alpha+\beta\in R$.
\end{proof}

\begin{lemma}\label{lem:zow}
The conormal bundle $T_X^*C_w\rightarrow C_w$ is isomorphic to the vector bundle $BwB\times^B\u_w\rightarrow C_w$, given by $(bw,x)\mapsto bw(\mod P)$. 
\end{lemma}
\begin{proof}
Let $\operatorname{pr}:X^B_w\rightarrow X_w$ be restriction of the quotient map $\nicefrac GB\rightarrow\nicefrac GP$ to $X^B_w$.
Since $w\in W^P$, the map $\operatorname{pr}$ restricts to an isomorphism of Schubert cells $C^B_w\xrightarrow\sim C_w$, see \cite{MR1923198}.
The claim now follows from the observation (cf. \cite[\S4.3]{conormal2017}) that for any $(bw,x)\in T^*X$, we have $(bw,x)\in T^*_XC_w$ if and only if $x\in\u_w$.
\end{proof}

\subsection{The Subgroup $Q$}
\label{defn:zwq}
As a consequence of the previous lemma, we see that $Stab_G(\u_w)$ is a standard parabolic subgroup.
Let $Q$ be any standard parabolic subgroup contained in $P$ that stabilizes $\u_w$, i.e.,\begin{align}\label{defn:Q}
    Q\subset Stab_G(\u_w)\cap P.
\end{align}
We define a vector bundle $\pi^Q_w:Z_w^Q\rightarrow X_w^Q$, where,\begin{align*}
    Z_w^Q=BwQ\times^Q\u_w,	&&\pi^Q_w(g,x)=g(\mod Q).
\end{align*}

Let $\underline w=(s_1,\cdots,s_r)$ be a minimal word for $w$. 
Recall the Bott-Samelson variety \bsw from \cref{bsw}. 
We lift $Z_w^Q$ to a vector bundle $\pi_{\underline w}:\reso\rightarrow\bsw$ given by,
\begin{align*} 
    \reso=P_1\times^B\cdots\times^BP_r\times^B\u_w, 
\end{align*}
and $\pi_{\underline w}(\underline p,x)=\underline p(\mod B)$, for $\underline p\in P_1\times^B\cdots\times^BP_r$.

\begin{proposition}\label{tau}
Let $\tau$ be the quotient map $G\times^Q\u\rightarrow G\times^P\u$. 
Viewing $Z^Q_w$ as a subvariety of $G\times^Q\u$, we have $\tau(Z_w^Q)\subset T_X^*X_w$. 
Let $\tau_w:Z_w^Q\rightarrow T_X^*X_w$ denote the induced map.
We have a commutative diagram,\[
\begin{tikzcd}
    \reso\arrow[d,"\pi_\w"]\arrow[r,"\theta^Q_\w"]  &Z_w^Q\arrow[r,"\tau_w"]\arrow[d,"\pi^Q_w"] &T_X^*X_w\arrow[d,"\pi_w"]\\
    \bsw\arrow[r,"\rho^Q_\w"]                       &X_w^Q\arrow[r,"\operatorname{pr}"]         &X_w
\end{tikzcd}
\]
Here $\operatorname{pr}:X^Q_w\rightarrow X_w$ is the restriction of the quotient map $\nicefrac GQ\rightarrow\nicefrac GP$ to $X^Q_w$, and $\theta^Q_\w:\reso\rightarrow Z_w^Q$ is the map given by $\theta^Q_\w(p_1,\cdots,p_r,x)=(p_1\cdots p_r,x)$.
\end{proposition}
\begin{proof}
Let $(Z^Q_w)^\circ$ be the restriction of of the vector bundle $\pi^Q_w:Z^Q_w\rightarrow X^Q_w$ to the Schubert cell $C_w^Q$.
The quotient map $\nicefrac GB\rightarrow\nicefrac GQ$ induces an isomorphism $C_w^B\xrightarrow\sim C_w^Q$ of Schubert cells.
Consequently, the quotient map,\begin{align}\label{work50}
	Z^\circ_w=BwB\times^B\u_w\longrightarrow BwQ\times^Q\u_w,	&&(bw,x)\mapsto (bw,x),
\end{align}
is an isomorphism. 
Observe that this map is the inverse of $\tau|(Z^Q_w)^\circ$, and so, $$\tau((Z^Q_w)^\circ)=T^*_{X}C_w\subset T_X^*X_w.$$
Now, since $T_X^*X_w$ is a closed subvariety, it follows that $\tau(Z^Q_w)\subset T_X^*X_w$. 

Finally, the commutativity of the diagram is a simple verification based on the formulae defining the various maps.
\end{proof}

Before we prove \cref{bir0}, let us recall some standard results about proper maps, which the reader can find, for example, in \cite[Ch.2]{MR0463157}.
\begin{proposition}\label{proper0}
The following properties are true:\begin{enumerate}
    \item Closed immersions are separated and proper.
    \item The composition of proper maps is proper.
    \item If $g:X\rightarrow Y$ is a proper map, then $g\times\mathrm{id}_Z:X\times Z\rightarrow Y\times Z$ is proper.
    \item Let $f:Y\hookrightarrow Z$ be a closed immersion.
    A map $g:X\rightarrow Y$ is proper if and only if $f\circ g$ is proper.
\end{enumerate}
\end{proposition}

\begin{mainTheorem}\label{bir0}
The maps $\theta^Q_\w$ and $\tau_w$ are proper and birational, and the composite map $\theta_\w\define\tau_w\circ\theta_w^Q$ is a $B$-equivariant resolution of singularities $\theta_\w:\reso\rightarrow\con$.
The map $\theta_\w$ is independent of the choice of $Q$.
\end{mainTheorem}
\begin{proof}
The birationality of $\tau_w$ is a consequence of \cref{work50}.
Recall from \cref{bsw} that $\rho^Q_\w$ induces an isomorphism $(\bsw)^\circ\xrightarrow\sim C^Q_w$.
Consequently, $\theta^Q_\w$ induces an isomorphism $\resO\xrightarrow\sim(Z^Q_w)^\circ$.
It follows that $\theta^Q_\w$ is birational.

Consider now the commutative diagram\begin{equation}\label{diag2}
\begin{tikzcd}[column sep=large]
    \reso\arrow[d,hook,"f"]\arrow[r,"\theta^Q_\w"]                  &Z_w^Q\arrow[r,"\tau_w"]\arrow[d,hook,"g"]                                  &T_X^*X_w\arrow[d,hook,"h"]\\
    \bsw\times\Ni\arrow[r,"\rho^Q_\w\times\operatorname{id}_{\Ni}"] &X_w^Q\times\Ni\arrow[r,"\operatorname{pr}\times\operatorname{id}_{\Ni}"]   &X_w\times\Ni,
\end{tikzcd}
\end{equation}
where $f,g,h$ are the closed immersions given by\begin{align*}
    f(p_1,\cdots,p_r,x) &=(\pi_\w(p_1,\cdots,p_r,x),Ad(p_1\cdots p_r)x),\\
    g(a,x)              &=(\pi^Q_w(a,x),Ad(a)x),\\
    h(a,x)              &=(\pi_w(a,x),Ad(a)x).
\end{align*}
Observe that the map, $$(\operatorname{pr}\times\operatorname{id}_{\Ni})\circ(\rho^Q_\w\times\operatorname{id})=\rho_\w\times\operatorname{id}_{\Ni},$$
is independent of the choice of $Q$, and therefore, the map $\theta_\w=\tau_w\circ\theta_\w^Q$ is also independent of the choice of $Q$.

Next, the maps $\rho^Q_\w$ and $\operatorname{pr}$ are proper; hence $\rho^Q_\w\times\operatorname{id}_{\Ni}$ and $\bsw\times\Ni$ are proper.
Consequently, $\theta_\w^Q$ and $\tau_w$ are proper.

Finally, observe that ${\widetilde{Z_{\w}}}$, being a vector bundle over the smooth variety $\bsw$, is itself a smooth variety.
Therefore, the map $\theta_\w$ is a resolution of singularities.
\end{proof}

%% file: sln.tex
\section{The Type A Grassmannian}
\label{sec:sln}
In this section, we recall the classical theory of Schubert varieties in type A. 
Further, for any Schubert variety $X_w\subset Gr(d,n)$, we choose a particular standard parabolic subgroup $Q$ satisfying \cref{defn:Q}. 
This choice of $Q$ allows us to present a uniform proof of \cref{finalEquations} in \cref{sec:equations}.
The primary reference for this section is \cite{MR2388163}.

\subsection{The Root System of $SL_n$}
\label{rootSystemA}
Let $\E n.$ be a $n$--dimensional vector space with privileged basis $\{e_1,\cdots,e_n\}$. 
The group $G=SL_n$ acts on \E n. by left multiplication with respect to the basis $e_1,\cdots,e_n$.
The Lie algebra \lgl of $G$ is precisely the set of traceless $n\times n$ matrices, i.e.,\[
    \lgl=\set{x\in\operatorname{Mat}_n(\k)}{\operatorname{trace}(x)=0}.
\]
Let $\mathfrak t$ be the set of diagonal matrices in $\mathfrak g$. 
For $1\leq i\leq n$, let $\epsilon_i\in\mathfrak t^*$ be the linear functional given by \[
    \left\langle\epsilon_i,\sum\limits_{j=1}^n a_jE_{j,j}\right\rangle=a_i,
\]
where $E_{j,j}$ is the diagonal matrix with entry $1$ in the $j^{th}$ position and zero elsewhere.

We fix $B$ to be the set of upper triangular matrices in $G$. 
The Lie algebra $\mathfrak b$ of $B$ is then precisely the set of upper triangular matrices in \lgl.
The root system $\roots..$ of $\mathfrak g$ with respect to $(\mathfrak b,\mathfrak t)$ is precisely,\begin{align*}
	\roots..&=\left\{\epsilon_i-\epsilon_j\mid\,1\leq i\neq j\leq n\right\},
\end{align*}
with the simple root $\alpha_i\in S=\mathsf A_{n-1}$ being given by $\alpha_i=\epsilon_i-\epsilon_{i+1}$.
The root $\epsilon_i-\epsilon_j$ is positive if and only if $i<j$. 

We denote by $E_{i,j}$ the elementary $n\times n$ matrix with $1$ in the $(i,j)$ position and $0$ elsewhere; and by $[E_{i,j}]$ the one-dimensional subspace of \lgl spanned by $E_{i,j}$.
Then $[E_{i,j}]$ is precisely the root space corresponding to the root $\epsilon_i-\epsilon_j$.

\subsection{Partial Flag Varieties}
Let $\underline q=(q_0,\cdots,q_r)$ be an integer-valued sequence satisfying $0=q_0\leq q_1\leq\cdots\leq q_r=n$. 

For $0\leq i\leq n$, we denote by \E i., the subspace of \E n. with basis $e_1,\cdots,e_i$, and by $\E\underline q.$, the partial flag $\E q_0.\subset\cdots\subset\E q_r.$. 
Let $Q$ be the parabolic subgroup of $SL_n$ corresponding to the subset $S_Q=\set{\alpha_j}{j\neq q_i,\,1\leq i\leq r}$.
The variety $X^Q=\nicefrac GQ$ is precisely the variety of \emph{partial flags of shape $\underline q$},\begin{equation*}\label{flag1}
    X^Q=\set{\F q_0.\subset\cdots\subset\F q_r.}{\dim\F q_i.=q_i}.
\end{equation*}
For brevity, we will denote a partial flag $\F q_0.\subset\cdots\subset\F q_r.$ of shape $\underline q$ by $\F\underline q.$.

As a particular example, let $P$ be the standard parabolic subgroup corresponding to the subset $S\backslash\{\alpha_d\}$. 
Then $X^d:=\nicefrac GP$ is precisely,
\begin{align*}
	X^d	=Gr(d,n)=\left\{V\mid\dim V=d\right\}. 
\end{align*}

\subsection{The Weyl Group}
The Weyl group of $G$ is isomorphic to $S_n$, the symmetric group on $n$ elements. 
The action of $W$ on \roots.. is given by the formula, $$w(\epsilon_i-\epsilon_j)=\epsilon_{w(i)}-\epsilon_{w(j)}.$$
In particular, $w(\epsilon_i-\epsilon_j)>0$ if and only if $w(i)<w(j)$.

The set of \emph{minimal representatives} with respect to $Q$ is given by,\begin{align}\label{sln:minRep}
    S_n^Q=\set{w\in S_n}{w(q_i+1)<w(q_i+2)<\cdots<w(q_{i+1}),\forall\,0\leq i\leq r}.
\end{align}

\subsection{Schubert Varieties} 
For $w\in S_n$, let $\mw (i,j)$ be the number of non-zero entries in the top left $i\times j$ sub-matrix of the permutation matrix $\sum E_{w(k),k}$, i.e., \begin{equation}\label{defn:bi}\begin{split}
	\mw(i,j)&\define\#\left\{w(1),\cdots,w(j)\right\}\cap\left\{1,\cdots,i\right\}\\
			&=\#\set{(k,w(k))}{k\leq j, w(k)\leq i}.
\end{split}
\end{equation}

The Schubert cells, $C^Q_w\subset X^Q$ and $C^d_w\subset X$, are given by,\begin{align*}
	C_w^Q&=\set{\F\underline q.\in X^Q}{\dim(F(q_i)\cap E(j))=\mw (j,q_i),\,1\leq j\leq n,1\leq i\leq l},\\
	C^d_w&=\set{V\in X^d}{\dim(V\cap E(j))=\mw (j,d),\,1\leq j\leq n},
\end{align*}
while the Schubert varieties, $X_w^Q\subset X^Q$ and $X^d_w\subset X^d$, are given by
\begin{equation}\label{def:Xw}\begin{split}
	X^Q_w	&=\set{\F\underline q.\in X^Q}{\dim\left(\F q_i.\cap\E j.\right)\geq\mw (j,q_i),\,1\leq j\leq n,1\leq i\leq l},\\
	X^d_w	&=\set{V\in X}{\dim\left(V\cap\E j.\right)\geq \mw (j,d),\,1\leq j\leq n}.
\end{split}
\end{equation}
In particular, we have $\F\underline q.\in X^Q_w$, if and only if $\F q_i.\in X^{q_i}_w$ for all $i$.

\subsection{The Projection Map}\label{proMap}
Suppose $d=q_i$ for some $i$, and consider the projection map 
$\operatorname{pr}_d:X^Q\rightarrow X^d$,
given by 
$\F\underline q.\mapsto \F d.$.

For any $w\in S_n$, we have $\operatorname{pr}_d(X_w^Q)=X^d_w$.
Further, if $w$ satisfies\begin{equation}\begin{split}\label{maxRep}
	w(1)>\cdots>w(d),\\
	w(d+1)>\cdots>w(n),
\end{split}
\end{equation}
then $\operatorname{pr}_d^{-1}(X_w^d)=X_w^Q$, see \cite{MR2388163}.
Any $w\in S_n$ satisfying \Cref{maxRep} is called a \emph{maximal representative} with respect to $P$.

\input{lemmae.tex}

%
\subsection{The Numbers $r_i,c_i$}\label{rici}\label{desc:w}
{\def\d{$\ddots$}
	For integers $a,b$, let $(a,b]$ denote the sequence,\[a+1,a+2,\cdots,b.\]
We fix $w\in\W^P$.
Following \cref{sln:minRep}, we have,
\begin{align*}
    w(1)<w(2)<\ldots<w(d),  &&w(d+1)<\ldots<w(n).
\end{align*}
Consequently, $w\in S_n^P$ is uniquely identified by the sequence $w(1),\cdots,w(d)$, which we now write as the following concatenation of contiguous sequences,
\begin{align*}
	(t'_1,t_1],\ (t'_2,t_2],\ \cdots,\ (t'_l,t_l].
\end{align*}
Here the $t_i,t'_i$ are certain integers satisfying,\[
	0\leq t'_1<t_1<t'_2<\cdots<t_{l-1}<t'_l<t_l\leq n, 
\]
and $\sum(t_i-t'_i)=d$.
For convenience, we set $t_0=0$ and $t'_{l+1}=n$.
The sequence $w(d+1),\cdots,w(n)$ is precisely,
\begin{align*}
	(t_0,t'_1],\ (t_1,t'_2],\ \cdots,\ (t_l,\cdots,t'_{l+1}].
\end{align*}
Consider the partial sums $r_0,\cdots,r_l$, and $c_0,\cdots,c_{l+1}$, given by,\begin{equation}\label{defn:rc}
    r_i\define\sum\limits_{1\leq j\leq i}(t_j-t'_j),\qquad\qquad      
    c_i\define\sum\limits_{1\leq j\leq i}(t'_j-t_{j-1}). 
\end{equation}
For $1\leq i\leq l$, we have $t_i=\up ri+\up ci$.
Further, 
we have $r_l=d$, $c_{l+1}=n-d$, and\begin{equation}\label{eq:mw}\begin{split}
    \mw (t_i,r_j)    &=\min\{r_i,r_j\}=r_{\min\{i,j\}},\\
    \mw (t_i,d+c_j)  &=r_i+\min\{c_i,c_j\}=r_i+c_{\min\{i,j\}},
\end{split}
\end{equation}
for all $1\leq i,j\leq l$.
The permutation matrix of $w$ is described in \cref{fig:perm}.
\begin{figure}
\begin{tabular}{|c|c|c|c|c|c|c|}\hline
	&   &  	&\d	&	& 	&	\\\hline
\d	&	&	& 	&  	& 	&	\\\hline
	&	&   &	&\d	& 	&	\\\hline
	&\d	&	&	& 	& 	&	\\\hline
	&   &   &	&	&\d	&	\\\hline
	&   &\d	&	&	& 	&	\\\hline
	&   &  	&	&	& 	&\d	\\\hline
\end{tabular}
\caption{The permutation matrix of $w$. 
	The empty boxes are zero matrices, while the dotted cells are identity matrices of size $r_1,r_2-r_1,\cdots,r_l-r_{l-1},c_1,c_2-c_1,\cdots,c_{l+1}-c_l$, going left to right. 
	}
\label{fig:perm}
\end{figure}
}

\begin{proposition}\label{choice:QA}
Consider the sequence $\underline q=(q_0,\cdots,q_{2l+1})$ given by\begin{align*}\label{defn:q}
    q_i=\begin{cases}\up ri &\text{for }0\leq i\leq l,\\
                d+\up c{i-l}&\text{for }l<i\leq 2l+1.
        \end{cases}
\end{align*}
Let $Q$ be the standard parabolic subgroup associated to the sequence $\underline q$, in the sense of \cref{flag1}.
Then $Q$ satisfies \cref{defn:Q}, i.e., $Q\subset Stab_G(\u_w)\cap P$.
\end{proposition}
\begin{proof}
Observe that $q_l=d$; hence, we have $Q\subset P$.
It remains to show that $Q$ stabilizes $\u_w$.
Recall the set $R$ and the subspace $\u_w$ from \cref{defn:uw}.
It follows from \cref{desc:w} that\begin{equation}\label{uw:sln}\begin{split}
    R   &=\set{\epsilon_i-\epsilon_j}{\exists k\text{ such that }i\leq q_k,j>q_{k+l}},\\ 
    \u_w&=\set{x\in\lgl}{x\E q_{l+i}.\subset\E q_{i-1}.,\,\forall\,1\leq i\leq l+1}.
\end{split}
\end{equation}
Now, since $Q$ stabilizes the flag $\E\underline q.$, it also stabilizes $\u_w$.
\end{proof}

%% file: lemmae.tex
The following lemmas are easy consequences of standard results on Schubert varieties.
They are used repeatedly in the proofs of \cref{rows,cols,orbitalConormal}.

\begin{lemma}\label{weak1}
Suppose we have integers $0\leq k\leq d\leq n$, a permutation $w\in S_n$, and a $k$-dimensional subspace $U\subset\E n.$. 
If $U\in X^k_w$, then 
\begin{align*}\hspace{40pt}
	\dim(U\cap\E i.)&\geq\mw(i,d)-(d-k)	&&\forall\,1\leq i\leq n.
\end{align*}
Conversely, suppose the above inequalities hold, and further, $w(1)>\cdots>w(d)$. 
Then $U\in X_w^k$.
\end{lemma}
\begin{proof}
Observe that for $1\leq i\leq n$, we have,\begin{align*}\hspace{40pt}
	\mw(i,k)\geq\max\{0,\mw(i,d)-(d-k)\},
\end{align*}
with equality holding for all $i$ if and only if $w(1)>\cdots>w(d)$. 
\end{proof}

\begin{lemma}\label{weak2}
Consider integers $0\leq k\leq d\leq n$, and a permutation $w\in S_n$, satisfying $w(k+1)>\cdots>w(n)$. 
Given $U\in X^d$, we have $U\in X^d_w$, if and only if,\begin{align*}\hspace{40pt}
	\dim(U\cap\E i.)&\geq\mw(i,k)	&&\forall\,1\leq i\leq n.
\end{align*}
\end{lemma}
\begin{proof}
There exists $1\leq i\leq n$ such that $\dim(U\cap \E i.)=k$.
Set $V=U\cap\E i.$. 
It is clear that $V\in X^k_w$. 
Now, since the statement of the lemma only involves $w$ via the integers $\mw(i,k)$ and $\mw(i,d)$, we may assume without loss of generality that $w(1)>\cdots>w(k)$.

Let $Q$ be the parabolic group corresponding to the sequence $\underline q=(k,d)$.
Then, we have $\operatorname{pr}_k(V\subset U)=V$ and $\operatorname{pr}_d(V\subset U)=U$.
It follows from \Cref{proMap} that,\[
	\operatorname{pr}_k^{-1}(X_w^k)=X_w^Q\implies\operatorname{pr}_d(\operatorname{pr}_k^{-1}(X_w^k))=X_w^d.
\]
Consequently, we obtain $U\in X^d_w$.
\end{proof}

\begin{proposition}
\label{strong}
Consider integers $0\leq k\leq d\leq m\leq n$, and a permutation $w\in S_n$.
Given subspaces $U\subset V\subset\E n.$ satisfying $\dim U=k$, and 
\begin{align*}\hspace{40pt}
	\dim(U\cap\E i.)&\geq\mw(i,d)-(d-k) &&\forall\,1\leq i\leq n,\\
    \dim(V\cap\E i.)&\geq\mw(i,m)   &&\forall\,1\leq i\leq n,
\end{align*}
there exists $U'\in X_w^d$ satisfying $U\subset U'\subset V$.
\end{proposition}
\begin{proof}
Set $l=\dim V$. 
Observe that\[
	l=\dim (V\cap\E n.)\geq\mw(n,m)=m.
\]
Let $Q'$ be the parabolic group corresponding to the sequence $(k,l)$, and $Q$ the parabolic group corresponding to the sequence $(k,d,l)$. 
We have a projection map $\operatorname{pr}:X^Q\rightarrow X^{Q'}$, given by $\F k,d,l.\mapsto\F k,l.$.

Since the statement of the proposition only involves $w$ via the integers $\mw(i,d)$ and $\mw(i,m)$, we may replace $w$ by any permutation $v$ satisfying\begin{align*}
	\{v(1),\cdots,v(d)\}	&=\{w(1),\cdots,w(d)\},\\
	\{v(d+1),\cdots,v(m)\}	&=\{w(d+1),\cdots,w(m)\},\\
	\{v(m+1),\cdots,v(n)\}	&=\{w(m+1),\cdots,w(n)\},
\end{align*}
without changing the statement.
In particular, we may assume that \begin{equation}\begin{split}\label{work60}
	w(1)>w(2)>\cdots>w(d),\\
	w(d+1)>\cdots>w(m),\\
	w(m+1)>\cdots>w(n).
\end{split}
\end{equation}
Using \Cref{weak1,weak2}, we deduce that $(U\subset V)\in X^{Q'}_w$.
Further, it follows from \Cref{work60} that the projection map $X^Q_w\rightarrow X^{Q'}_w$ is surjective, see \cite{MR2388163}. 
In particular, there exists a partial flag $\F k,d,l.\in X^Q_w$, for which $\F k.=U$ and $\F l.=V$.
This partial flag yields the required subspace $\F d.$, satisfying $U\subset\F d.\subset V$, and $\F d.\in X^d_w$.
\end{proof}

%% file: typeC.tex
\section{The Symplectic Grassmannian}
\label{sec:C}
In this section, we recall some facts about Schubert subvarieties of the symplectic Grassmannian.
In particular, for any Schubert variety $X_w$ in the symplectic Grassmannian, we choose a particular standard parabolic subgroup $Q$ satisfying \cref{defn:Q}. 
This choice of $Q$ allows us to present a uniform proof of \cref{finalEquations} in \cref{sec:equations}.
The primary reference for this section is \cite{MR2388163}.

\subsection{The Bilinear Form}
Let $\E 2d.$ be a $2d$--dimensional vector space with a privileged basis $\{e_1,\cdots,e_{2d}\}$. 
For $1\leq i\leq 2d$, we define,$$\overline i\define 2d+1-i.$$
Consider the non-degenerate skew-symplectic bilinear form $\ws$ on \E 2d. given by,\begin{align*}
	\ws(e_i,e_j)=\begin{cases}\delta_{i,\overline j}&\text{if }i\leq d,\\
					-\delta_{i,\overline j}&\text{if }i>d.
				\end{cases}
\end{align*}
For $V$ a subspace of \E 2d., we define,\[
	V^\perp=\set{u\in\E 2d.}{\omega(u,v)=0,\ \forall\,v\in V}.
\]	
A simple calculation yields $\E i.^\perp=\E{2d-i}.$, for $1\leq i\leq 2d$.
Further, as a consequence of the non-degeneracy of $\ws$, we have the formulae,\begin{align}\label{eq:perp}
	\dim V+\dim V^\perp	=2d,&&
	U^\perp\cap V^\perp	=(U+V)^\perp,
\end{align}
and $(V^\perp)^\perp=V$, for any subspaces $U,V\subset\E 2d.$.

\subsection{The Symplectic Group $Sp_{2d}$}\label{lglC}
Let $G=Stab_{SL_{2d}}(\ws)$, i.e.,\[
	G=\set{g\in SL_{2d}}{\omega(gu,gv)=\omega(u,v),\,\forall\,u,v\in\E 2d.}.
\]
The group $G$ is the symplectic group $Sp_{2d}$. 
Its Lie algebra \lgl is given by\[
	\lgl=\set{x\in\mathfrak{sl}_{2d}}{\omega(xu,v)+\omega(u,xv)=0,\ \forall\,u,v\in\E 2d.}.
\]
Let $T'$ (resp. $\mathfrak t'$) be the set of diagonal matrices, and $B'$ (resp. $\mathfrak b'$) the set of upper triangular matrices in $SL_{2d}$ (resp. $\mathfrak{sl}_{2d}$).
The subgroup $T=T'\cap G$ is a maximal torus in $G$, and the subgroup $B=B'\cap G$ is a Borel subgroup of $G$.

\subsection{The Root System of $Sp_{2d}$} 
The group $G$ is a simple group with Dynkin diagram $\mathsf C_d$.
Recall from \cref{rootSystemA}, the linear functionals $\epsilon_1,\cdots,\epsilon_{2d}$ on $\mathfrak t'$.
By abuse of notation, we also denote by $\epsilon_i$, the restriction $\epsilon_i|\mathfrak t$.

Following \cite{MR2388163}, we present the root system of $G$ with respect to $(B,T)$.
The simple root $\alpha_i\in S=\mathsf C_d$ is given by \begin{align*}
	\alpha_i=\begin{cases}\epsilon_i-\epsilon_{i+1}&\text{for }1\leq i<d,\\ 2\epsilon_d&\text{for }i=d.\end{cases}
\end{align*}
The set of roots \roots.., and the set of positive roots \roots+.., are given by
\begin{align*}
	\roots..	&=\set{\pm\epsilon_i\pm\epsilon_j}{1\leq i\neq j\leq d}\sqcup\set{\pm 2\epsilon_i}{1\leq i\leq d},\\
	\roots+..	&=\set{\epsilon_i\pm\epsilon_j}{1\leq i<j\leq d}\sqcup\set{2\epsilon_i}{1\leq i\leq d}.
\end{align*}
The corresponding root spaces are given by $\lgl_{2\epsilon_i}=[E_{i,\overline i}]$, $\lgl_{-2\epsilon_i}=[E_{\overline i,i}]$, 
\begin{align*}
	\lgl_{\epsilon_i+\epsilon_j}=[E_{i,\overline j}+E_{j,\overline i}],	&&\lgl_{-\epsilon_i-\epsilon_j}=[E_{\overline i,j}+E_{\overline j,i}],	&&\lgl_{\epsilon_i-\epsilon_j}=[E_{i,j}-E_{\overline j,\overline i}].
\end{align*}

\subsection{The Weyl Group}
Let $s_1,\cdots,s_d$ denote the simple reflections in Weyl group $W$ of $G$, and let $r_1,\cdots,r_{2d-1}$ denote the simple reflections of $S_{2d-1}$.
We have an embedding $W\hookrightarrow S_{2d}$, given by,\[
	s_i\mapsto\begin{cases}
		r_ir_{2d-i}	&\text{for }1\leq i<d,\\
		r_d			&\text{for }i=d.
					\end{cases}
\]
Via this embedding, we have,
\begin{align*}
	W=\set{w\in S_{2d}}{\overline{w(i)}=w(\overline i),\,1\leq i\leq d}.
\end{align*}

The Bruhat order on $S_{2d}$ induces a partial order on $W$.
This induced order is precisely the Bruhat order on $W$.
Further, by virtue of being a subgroup of $S_{2d}$, the group $W$ acts on $\mathfrak{sl}_{2d}$.
One obtains the action of $W$ on \lgl by restricting this action.

%
\subsection{Standard Parabolic Subgroups}
\label{Cspfv}
Let $\underline q=(q_0,\cdots,q_r)$ be any integer-valued sequence satisfying $0=q_0\leq q_1\leq\cdots\leq q_r=2d$, and further, $q_i+q_{r-i}=2d$ for $1\leq i\leq r$. 
Suppose $Q'$ is the standard parabolic subgroup of $SL_{2d}$ corresponding to the subset,
$$\set{\alpha_j\in\mathsf A_{2d-1}}{j\neq q_i,\,1\leq i\leq r-1}.$$ 
Then $Q=Q'\cap G$ is the parabolic subgroup of $G$ corresponding to the subset,\[
	\set{\alpha_j\in S}{j\neq q_i,\,1\leq i\leq \left\lceil\nicefrac r2\right\rceil}.
\]
The variety $X^Q=\nicefrac GQ$ is precisely the variety of isotropic flags of shape $\underline q$, i.e.,
\begin{equation*}
	X^Q=\set{\F\underline q.\in\nicefrac{SL_{2d}}{Q'}}{\F q_i.^\perp=\F q_{r-i}.}.
\end{equation*}

As a particular example, let $P'$ be the standard parabolic subgroup of $SL_n$ corresponding to the subset $\mathsf A_{2d-1}\backslash\{\alpha_d\}$, and let $P=P'\cap G$. 
Then $P$ is the standard parabolic corresponding to $S\backslash\{\alpha_d\}$, and further,
\begin{align*}
	X	&=\nicefrac GP=\left\{V\subset\E 2d.\mid V=V^\perp\right\}. 
\end{align*}
Observe that the condition $V=V^\perp$ ensures that $\dim V=d$, see \cref{eq:perp}.

\subsection{Schubert Varieties} 
Consider an element $w\in W$.
By viewing $w$ as an element of $S_{2d}$, we define the numbers $\mw (i,k)$ precisely as in \cref{defn:bi}.
The Schubert cells $C_w$, $C^Q_w$ are then given by,\begin{align*}
	C_w^Q   &=\set{\F\underline q.\in X^Q}{\dim(F(q_i)\cap E(j))=\mw (j,q_i),\,1\leq i\leq l,1\leq j\leq n}, \\
	C_w   	&=\set{V\in X}{\dim(V\cap E(j))=\mw (j,d),\,1\leq j\leq n},
\end{align*}
and the Schubert varieties $X_w$ and $X_w^Q$ are given by,\begin{equation}\label{defCD:Xw}\begin{split}
	X^Q_w	&=\set{\F\underline q.\in X^Q}{\dim\left(\F q_i.\cap\E j.\right)\geq\mw (j,q_i),\,1\leq i\leq l,1\leq j\leq n},\\
	X_w		&=\set{V\in X}{\dim(V\cap\E i.)\geq\mw(i,d),\,1\leq i\leq n}.
\end{split}
\end{equation}
In particular, any Schubert subvariety of $X^Q$ can be identified (set-theoretically) as the intersection of a Schubert subvariety of $\nicefrac{SL_{2d}}{Q'}$ with $\nicefrac{Sp_{2d}}Q\subset\nicefrac{SL_{2d}}{Q'}$.

\subsection{Numerical Redundancy}
By viewing $w$ as an element of $S_{2d}$, we define the numbers $t_i,t'_i,r_i,c_i$ exactly as in \cref{desc:w,defn:rc}.
Observe that since $w(\overline i)=\overline{w(i)}$ for all $1\leq i\leq 2d$, the permutation matrix of $w$ is symmetric across the anti-diagonal, see \cref{fig:perm}.
Consequently, for any $0\leq i\leq l$, we have,\begin{align}\label{tiflip}
	r_i+c_{l-i}=d,	&&t_i+t_{l-i}=2d.
\end{align}
In particular, we have $\E t_i.^\perp=\E t_{l-i}.$.

The conditions defining the Schubert variety $X_w^Q\subset X^Q$, described in \cref{defCD:Xw}, are not minimal.
We describe this redundancy in the next lemma.

\begin{lemma}\label{checkHalf}
Consider $\F\underline q.\in X^Q$.
Then $\F\underline q.\in X^Q_w$ if and only if 
\begin{align*}\hspace{40pt}
	\dim\left(\F q_i.\cap\E j.\right)\geq\mw (j,q_i),	&&\quad 1\leq i\leq l,\,1\leq j\leq 2d.
\end{align*}
\end{lemma}
\begin{proof}
Since the permutation matrix of $w$ is symmetric across the anti-diagonal, the number of non-zero entries in the top left $i\times j$ corner of $w$ equals the number of entries in the bottom right $i\times j$ corner.
Further, since each row and column of this matrix has precisely one non-zero entry, we have,
\begin{align*}
	\mw(i,j)=	&\#\set{(k,w(k)}{k>2d-j,w(k)>2d-i}\\
			=	&2d-\#\left(\set{(k,w(k))}{k\leq 2d-j}\cup\set{(k,w(k))}{w(k)\leq 2d-i}\right)\\
			=	&2d-((2d-j)+(2d-i)-\mw(2d-i,2d-j)).
\end{align*}
Hence, for $1\leq i,j\leq 2d$, we have the formula, 
\begin{align*}
	2d-(i+j-\mw(i,j))&=\mw(2d-i,2d-j).
\end{align*}
Consider some $\F\underline q.\in X^Q$ satisfying the inequalities of the lemma.
Given $1\leq i\leq l$ and $1\leq j\leq 2d$, we have,\begin{equation*}\begin{aligned}
	\dim(\F q_i.\cap\E j.)				&\geq\mw(j,q_i)\\
\implies\dim(\F q_i.+\E j.)				&\leq q_i+j-\mw(j,q_i)\\
\implies\dim((\F q_i.+\E j.)^\perp)		&\geq 2d-(q_i+j-\mw(j,q_i))\\
\implies\dim(\F q_i.^\perp\cap\E 2d-j.)	&\geq\mw(2d-j,2d-q_i).
\end{aligned}
\end{equation*}
The final inequality follows from the penultimate as a consequence of \cref{eq:perp}.
We see that $\F\underline q.$ satisfies \cref{defCD:Xw}, and hence obtain $\F\underline q.\in X_w^Q$.
\end{proof}

\subsection{The Subspace $\u_w$}
Let $\mathfrak v$ be the Lie algebra of the unipotent radical of $P'$, and \u the Lie algebra of the unipotent radical of $P$.
We have\begin{align*}
	\mathfrak v	&=\bigoplus\limits_{i\leq d<j}[E_{i,j}],&&\u=\bigoplus\limits_{1\leq i<j\leq d}\lgl_{\epsilon_i+\epsilon_j}=\bigoplus\limits_{1\leq i<j\leq d}[E_{i,\overline j}+E_{j,\overline i}].
\end{align*}
In particular, we have $\u=\mathfrak v\cap\lgl$. 
Recall the subspace $\u_w$ from \cref{defn:uw}.
Since \lgl is stable under the action of $Ad(w^{-1})$, we have,\begin{align}\label{uw:cd0}
	\u_w	&=\u\cap Ad(w^{-1})\mathfrak b=(\mathfrak v\cap\lgl)\cap Ad(w^{-1})(\mathfrak b'\cap\lgl)=\mathfrak v\cap Ad(w^{-1})\mathfrak b'\cap\lgl. 
\end{align}
Let $q_0,q_1,\cdots,q_{2l+1}$ be the sequence defined by\begin{align*}
    q_i=\begin{cases}\up ri &\text{for }0\leq i\leq l,\\
                d+\up c{i-l}&\text{for }l<i\leq 2l+1.
        \end{cases}
\end{align*}
It follows from \cref{tiflip} that $q_i+q_{2l-i}=2d$ for all $1\leq i\leq 2l$.

\begin{proposition}\label{choice:QC}
Let $Q$ be the standard parabolic subgroup of $G$ associated to the sequence $\underline q=(q_0,\cdots,q_{2l})$, in the sense of \cref{Cspfv}.
Then $Q$ satisfies \cref{defn:Q}, i.e., $Q\subset Stab_G(\u_w)\cap P$.
\end{proposition}
\begin{proof}
It follows from \cref{tiflip} that $c_{l+1}=c_l=d$, hence $q_{2l+1}=q_{2l}=2d$.
Therefore, the standard parabolic subgroup $Q'\subset SL_{2d}$ associated to $(q_0,\cdots,q_{2l+1})$ is the same as the standard parabolic subgroup of $SL_{2d}$ associated to $(q_0,\cdots,q_{2l})$.

Next, it follows from \cref{uw:cd0,uw:sln} that\begin{align}\label{uw:cd}
	\u_w=\set{x\in\lgl}{x\E q_{l+i}.\subset\E q_{i-1}.,\,\forall\,1\leq i\leq l+1}.
\end{align}
Now, since $Q'$ stabilizes $\u_w$, and since $Q=Q'\cap G$, we have $Q\subset Stab_G(\u_w)$.
Finally, since $q_l=d$, we have $Q\subset P$, hence $Q\subset Stab_G(\u_w)\cap P$.
\end{proof}

%% file: extension.tex
\section{Defining Equations for the Conormal Variety in Types A and C}
\label{sec:equations}
Fix integers $d<n$. 
Let $G$ be either $SL_n$ or $SO_{2d}$, let $B$ be the subgroup of upper triangular matrices in $G$, and let $P$ be the standard parabolic subgroup of $G$ corresponding to the subset $S\backslash\{\alpha_d\}$ of simple roots.
As discussed in \cref{sec:sln,sec:C}, the variety $X=\nicefrac GP$ is either the usual Grassmannian $Gr(d,n)$ or the symplectic Grassmannian $SGr(2d)$.

We fix a Schubert variety $X_w\subset X$ corresponding to some $w\in W^P$.
In this section, we prove \cref{finalEquations}, which gives a system of defining equations for the conormal variety $T_X^*X_w$ as a subvariety of $T^*X$.
Let $\pi:T^*X\rightarrow X$ be the structure map, and $\mu$ the Springer map.
\Cref{finalEquations} states that a point $p\in T^*X$ is in $T^*_XX_w$ if and only if $\pi(p)\in X_w$ and $\mu(p)$ satisfies \cref{bigshow}.

Recall the commutative diagram from \cref{tau}. 
We show in \Cref{eqrc} that for any point in $Z^Q_w$, its image under $\mu\circ\tau_w$ satisfies \cref{finalEquations}.
Conversely, we show in \cref{rows,cols} that any point in $T^*X$ lying over $X_w$, and further satisfying \cref{bigshow}, belongs to $\tau_w(Z^Q_w)=T_X^*X_w$.

\subsection{Combinatorial Description of $X_w$}\label{sch:cont}\label{num}
Fix $w\in W^P$.
Let the integers $\mw(i,j)$, $r_i$, and $c_i$ be as in \Cref{defn:bi,defn:rc} respectively.
It follows from \Cref{def:Xw,eq:mw,defCD:Xw} that $\F\underline q.\in X^Q_w$ if and only if
\begin{align*}\hspace{30pt}
	\dim(\F q_i.\cap\E t_j.)    &\geq\min\{r_i,r_j\}=r_{\min\{i,j\}}		&\forall\,1\leq i,j\leq l,\\ 
	\dim(\F q_{i+l}.\cap\E t_j.)&\geq r_j+\min\{c_i,c_j\}=r_j+c_{\min\{i,j\}}&\forall\,1\leq i,j\leq l.
\end{align*}
In particular, when $i=j$, this yields $\F q_i.\subset\E t_i.\subset\F q_{i+l}.$.

\subsection{The Cotangent Bundle}
Let $\pi:T^*X\rightarrow X$ be the structure map defining the cotangent bundle, and $\mu:T^*X\rightarrow\Ni$ the Springer map, see \cref{sub:springer}.
We identify the cotangent bundle $T^*X$ with its image under the closed embedding $(\pi,\mu):T^*X\hookrightarrow X\times\Ni$,\begin{align*}
	T^*X=\set{(V,x)\in X\times\mathcal N}{x\E n.\subset V,\,xV=0}.
\end{align*}

\subsection{The Variety $Z^Q_w$}\label{tvzwq}
For $G=SL_n$, let $\underline q$ and $Q$ be as in \cref{choice:QA}.
For $G=Sp_{2d}$, let $\underline q$ and $Q$ be as in \cref{choice:QC}.
Recall the variety $Z^Q_w$ from \cref{defn:zwq}, and the descriptions of $\u_w$ from \cref{uw:sln,uw:cd}.
Using the closed embedding $f$ from \cref{diag2}, we obtain,
\begin{equation}\label{quiver}
\begin{split}
	Z^Q_w&=\set{(\F\underline q.,x)\in X^Q_w\times\Ni}{x\F q_{i+l}.\subset\F q_{i-1}.,\,\forall\,1\leq i\leq l+1}.
\end{split}
\end{equation}

\Cref{finalEquations} states that given $(V,x)\in T^*X$, we have $(V,x)\in T^*_XX_w$, if and only if $V\in X_w$, and $x$ satisfies \cref{bigshow}.
The purpose of the following proposition is to show that \cref{bigshow} is necessary, i.e.,
if $(V,x)\in T^*_XX_w$, then $x$ satisfies \cref{bigshow}. 

\begin{proposition}\label{eqrc}
For any point $(\F\underline q.,x)\in Z^Q_w$, we have, for $1\leq j<i\leq l$,
\begin{equation*} 
\dim(\nicefrac{xE(\up ti)}{E(\up tj)})  \leq
    \begin{cases}\up r{i-1}-\up rj,\\
        \up ci-\up c{j+1}.
    \end{cases}
\end{equation*}
\end{proposition}
\begin{proof}
Consider $(\F\underline q.,x)\in Z^Q_w$, and integers $1\leq j<i\leq l$.
We see from \cref{sch:cont} that $\E t_i.\subset\F q_{i+l}.$, and from \cref{quiver} that $x\F q_{i+l}.\subset\F q_{i-1}.$. 
Consequently, we have $x\E t_i.\subset\F q_{i-1}.$, and hence,
\begin{align*}
    \dim(\nicefrac{x\E t_i.}{\E t_j.})  &\leq\dim(\nicefrac{\F q_{i-1}.}{\E t_j.})\\
                                        &=\dim\F q_{i-1}.-\dim(\F q_{i-1}.\cap\E t_j.)\\
                                        &\leq r_{i-1}-r_j,
\end{align*}
where the final inequality follows from \cref{num}.
Next, we see from \cref{sch:cont,quiver} that $x\F q_{j+l+1}.\subset\F q_j.\subset\E t_j.$. 
In particular, $\F q_{j+l+1}.$ is contained in the kernel of the map,\begin{align*}
	\F q_{i+l}.\rightarrow\nicefrac{\E n.}{\E t_j.},&&v\mapsto xv(\mod\E t_j.).
\end{align*}
Since the image of this map is precisely $\nicefrac{x\F q_{i+l}.}{\E t_j.}$, we have,\begin{align*}
    \dim(\nicefrac{x\F q_{i+l}.}{\E t_j.})  &\leq\dim\F q_{i+l}.-\dim\F q_{j+l+1}.\\
											&=q_{i+l}-q_{j+l+1}=c_i-c_{j+1}.
\end{align*}
Finally, since $\E t_i.\subset\F q_{i+l}.$, we deduce that $\dim(\nicefrac{x\E t_i.}{\E t_j.})\leq c_i-c_{j+1}$.
\end{proof}

%% file: suff.tex
The following two propositions lay the groundwork required to prove that \cref{bigshow} is sufficient.
\begin{proposition}\label{rows}
Consider $(V,x)\in X_w\times\Ni$ satisfying $\im x\subset V\subset\ker x$, and\begin{align}\label{work31}\hspace{50pt}
    \dim(\nicefrac{xE\left(t_i\right)}{E\left(t_j\right)})\leq\up r{i-1}-\up rj,&&\qquad 0\leq j<i\leq l+1.
\end{align}
Then, there exists a sequence of subspaces $V_0\subset\cdots\subset V_l=V$, satisfying,\begin{equation}\label{eqfr}
\begin{aligned}
	\dim V_i					&=q_i,								\\
	x\E\up t{i+1}.				&\subset V_i\subset\E\up ti.,\\
	\dim(V_i\cap\E t_j.)		&\geq\min\{r_i,r_j\}=\mw(t_j,q_i),   
\end{aligned}
\end{equation}
for all $1\leq i,j\leq l$.
\end{proposition}
\begin{proof}
Since $V\in X_w$, it follows from \cref{num} that $V_l=V$ satisfies \Cref{eqfr}.
We construct the subspaces $V_i$ inductively.
In particular, given subspaces $V_i,\cdots,V_l$ satisfying \Cref{eqfr}, we construct $V_{i-1}$. 

Applying \Cref{work31} with $j=i-1$, we have $x\E t_i.\subset \E\up t{i-1}.$.
Further, \Cref{eqfr} yields $x\E t_i.\subset x\E t_{i+1}.\subset V_i$.
Hence, we have,\begin{align*}
    x\E\up ti.\subset  V_i\cap\E t_{i-1}..
\end{align*}

Set $U_1=x\E t_i.$, and $U_2=V_i\cap\E t_{i-1}.$.
Applying \Cref{work31} with $j=0$, we see that $\dim U_1\leq r_{i-1}$.
Let $k=\up r{i-1}-\dim U_1$.

Observe that $U_1\cap\E\up tj.$ is the kernel of the quotient map $U_1\rightarrow\nicefrac{U_1}{\E\up tj.}$. 
Now, since $\dim(\nicefrac{x\E t_i.}{\E t_j.})\leq r_{i-1}-r_j$, we have, for $1\leq j\leq l$,\begin{align*}
    \dim(U_1\cap\E\up tj.)	&=\dim U_1-\dim\left(\nicefrac{U_1}{\E\up tj.}\right)\\
								&\geq(r_{i-1}-k)-(\up r{i-1}-\up rj)=\up rj-k\\
								&\geq\min\{r_{i-1},r_j\}-k\\
								&=\mw(t_j,q_{i-1})-k.
\end{align*}
On the other hand, observe that,\begin{align*}
	U_2\cap\E t_j.	    &=\begin{cases} V_i\cap\E t_{i-1}.   &\text{if }i\leq j,\\
				 			V_i\cap\E t_j.                  &\text{if }i>j,
			    		\end{cases}\\
\implies\dim(U_2\cap\E t_j.)&\geq\begin{cases}r_{i-1}   &\text{if }i\leq j,\\
				    				r_j                 &\text{if }i>j,
								\end{cases}\\
							&=\min\{r_{i-1},r_j\}=\mw(t_j,q_{i-1}).
\end{align*}
 
It now follows from \cref{strong} that there exists a subspace $V_{i-1}$ satisfying $x\E t_i.\subset V_{i-1}\subset U_2\subset \E t_{i-1}.$, and \cref{eqfr}.
\end{proof}

\begin{proposition}\label{cols}
Consider $(V,x)\in X_w\times\Ni$ satisfying $\im x\subset V\subset \ker x$, and\begin{align}\label{bigshow2}\hspace{50pt}
	\dim(\nicefrac{xE(\up ti)}{E(\up tj)})\leq\up ci-\up c{j+1},&&&\forall\,0\leq j< i\leq l+1.
\end{align}
Then, there exists a sequence of subspaces $V= V_l\subset\cdots\subset V_{2l+1}$, satisfying,\begin{equation}\label{eqfc}
\begin{aligned}
	\dim V_{l+i}			&=q_{l+i},			\\
	V_{l+i}					&\subset\ker x+\E t_{i}.,\\
	\dim(V_{l+i}\cap\E t_j.)&\geq r_j+\min\{c_i,c_j\}=\mw(t_j,q_{l+i}).  
\end{aligned}
\end{equation}
for all $1\leq i,j\leq l$.
\end{proposition}
\begin{proof}
Since $V\in X_w$, it follows from \cref{num} that $ V_l=V$ satisfies \Cref{eqfc}.
We construct the subspaces $ V_{l+i}$ inductively.
In particular, given subspaces $V_l,\cdots, V_{l+i-1}$ satisfying \cref{eqfc}, we construct $ V_{l+i}$. 

We see from \cref{eqfc} that $V_{l+i-1}\subset \ker x+\E\up t{i}.$. 
Set $
	U=\ker x+\E\up t{i}.
$. 
We first prove that,\begin{align}\label{work10}\hspace{20pt}
	\dim(U\cap\E t_j.)\geq r_j+\min\{c_{i},c_j\}=\mw(t_j,q_{l+i})	&&\forall\,1\leq j\leq l+1.
\end{align}
For $j\leq i$, we have $\E t_j.\subset U$, hence $U\cap\E t_j.=\E t_j.$. 
It follows that,\begin{align*}
	\dim(U\cap\E t_j.)=t_j&=r_j+c_j=r_j+\min\{c_i,c_j\}.
\end{align*}
For $j>i$, consider the map,\begin{align*}
	\E t_j.\rightarrow\nicefrac{x\E t_j.}{\E t_{i-1}.},&& v\mapsto xv(\mod\E t_{i-1}.).
\end{align*}
Since $x\E t_i.\subset\E t_{i-1}.$, the subspace $U\cap\E t_j.$ is contained in the kernel of this map. 
Further, \cref{bigshow2} states that $\dim(\nicefrac{x\E\up tj.}{\E\up ti.})\leq\up cj-\up c{i+1}$, hence\begin{align*}
	\dim(U\cap\E\up tj.)	&\geq\up tj-(\up cj-\up ci)\\
							&=\up rj+\up c{i}=r_j+\min\{c_{i},c_j\}.
\end{align*}
This finishes the proof of \Cref{work10}.
It now follows from \cref{work10,weak1,strong} that there exists a subspace $V_{l+i}$ containing $V_{l+i-1}$, and further satisfying \cref{eqfc}.
\end{proof}

\begin{mainTheorem}\label{finalEquations}
Consider $(V,x)\in T^*X$.
Then $(V,x)\in T_X^*X_w$ if and only if $V\in X_w$, and further, for all $1\leq j<i\leq l+1$, we have,
\begin{align}\label{bigshow}
    \dim(\nicefrac{x\E\up ti.}{\E t_j.})\leq\begin{cases}\up r{i-1}-\up rj,\\ \up ci-\up c{j+1}.\end{cases}
\end{align}
\end{mainTheorem}
\begin{proof}
Recall the map $\tau_w:Z^Q_w\rightarrow T_X^*X_w$ from \cref{tau}, given by,$$\tau_w(\F\underline q.,x)=(\F d.,x).$$
The map $\tau_w$ is proper and birational (see \cref{bir0}), hence surjective.
It follows that $(V,x)\in T_X^*X_w$ if and only if there exists $\F\underline q.\in X_w^Q$ such that $\F d.=V$, and $(\F\underline q.,x)\in Z^Q_w$.

Consider $(V,x)\in T_X^*X_w$.
It follows from \cref{bir0} that $V\in X_w$, and from \cref{eqrc} that \cref{bigshow} holds.
Conversely, consider $(V,x)\in T^*X$ satisfying $V\in X_w$, and \cref{bigshow}.
We will construct $\F\underline q.\in X_w^Q$ such that $(\F\underline q.,x)\in Z^Q_w$, and $\tau_w(\F\underline q.,x)=(V,x)$.

Using \cref{rows}, we construct subspaces $V_0,\cdots,V_l=V$ satisfying \cref{eqfr}. 
Similarly, we use \cref{cols} to construct subspaces $V_{l+1},\cdots,V_{2l+1}$ satisfying \cref{eqfc}.

Suppose first that $G=SL_n$.
We set,\[
	\F\underline q.=V_0\subset V_1\subset\cdots\subset V_{2l+1}.
\]
Observe that $\F d.=V_l=V$.
It follows from \cref{eqfr,eqfc} that $\F\underline q.\in X_w^Q$. 
Further, for $1\leq i\leq l+1$, we have,
\begin{align*}
	\F\up q{l+i}.\subset\ker x+\E\up t{i}.
\implies	x\F\up q{l+i}.	&\subset x\E\up t{i}.\subset\F\up q{i-1}.. 
\end{align*}
This is precisely the condition for $(\F\underline q.,x)$ to belong to $Z^Q_w$.

Suppose next that $G=Sp_{2d}$.
Let $\F\underline q.$ be the partial flag given by,\[
	\F q_i.=\begin{cases}V_i	&\text{for }i\leq l,\\
				V_{2l-i}^\perp	&\text{for }l<i\leq 2l.
		\end{cases}
\]
In particular, we have $\F d.=V_l=V$.
It follows from \cref{checkHalf,rows} that $\F\underline q.\in X^Q_w$.
It remains to show that $(\F\underline q.,x)\in Z^Q_w$.

For $0\leq i\leq l$, we have $x\E t_{i+1}.\subset V_i$, hence $\omega(x\E t_{i+1}.,V_i^\perp)=0$.
It follows from \cref{lglC} that
$\omega(\E t_{i+1}.,x(V_i^\perp))=0$, hence,
\begin{align*}
	&x\F q_{2l-i}.=x(V_i^\perp)\subset\E t_{i+1}.^\perp=\E t_{l-i-1}..
\end{align*}
The final equality is a consequence of \cref{tiflip}.
Substituting $i\mapsto l-i$ yields $x\E q_{l+i}.\subset\E t_{i-1}.$ for all $0\leq i\leq l$.
It follows that $(\F\underline q.,x)\in Z^Q_w$, hence $(\F d.,x)\in T^*_XX_w$.
\end{proof}

%% file: orbital.tex
\section{Orbital Varieties}

Let $G$, $B$ be as in the previous sections.
Consider a $G$-orbit closure $\Ni[\lambda]\subset\mathcal N$.
The irreducible components of $\Ni[\lambda]\cap\u_B$ are called orbital varieties.
Orbital varieties are closely related to the conormal varieties of Schubert varieties. 
\begin{proposition}[cf. \cite{MR672610}]\label{orbitalConormal}
Given a standard parabolic subgroup $Q$, and a Schubert variety $X_w^Q\subset X^Q$, the image of the conormal variety $T^*_{X^Q}X^Q_w$ under the Springer map $\mu^Q:T^*X^Q\rightarrow\Ni$ is an orbital variety.
\end{proposition}

For more details on the relationship between conormal varieties and orbital varieties, the reader may consult \cite{MR2510045,MR672610}.
Providing a combinatorial description of the inclusion order on orbital varieties, and providing the defining equations for an orbital variety viewed as a subvariety of $\u_B$, are both open problems in general.
For certain orbital varieties in types A, B, C (those corresponding to the nilpotent orbits satisfying $x^2=0$), these problems were solved in \cite{MR2134184,MR3717217}.

Suppose $G$ is either $SL_{2n}$ or $Sp_{2d}$, and $P$ is the standard parabolic group corresponding to $S\backslash\{\alpha_d\}$.
We derive, in \cref{mainOrbital}, 
a system of defining equations for orbital varieties of the form $\mu(T^*_{X}X_w)$.
This is an easy consequence of \cref{finalEquations}, and recovers some of the results of \cite{MR2134184,MR3717217}.

\begin{mainTheorem}\label{mainOrbital}
Let $G$, $B$, $P$, $X$, $w$, and $\mu$ be as in \cref{finalEquations}.
Then,\[
	\mu(T^*_XX_w)=\set{x\in\u_B}{x^2=0,\,\dim(\nicefrac{x\E t_i.}{\E t_j.})\leq\begin{cases}\up r{i-1}-\up rj,\\\up ci-\up c{j+1},\end{cases}\forall\,1\leq i<j\leq l}.
\]
\end{mainTheorem}
\begin{proof}
	Consider $x\in\u_B$ satisfying $x^2=0$, and\begin{align}\label{work40}\hspace{30pt}
	\dim(\nicefrac{x\E t_i.}{\E t_j.})\leq\begin{cases}\up r{i-1}-\up rj,\\
											\up ci-\up c{j+1},
										\end{cases}&&\forall\,1\leq j<i\leq l.
\end{align}
Substituting $j=0$ in \cref{work40}, we obtain,\begin{align*}
		\dim(\im(x|\E t_i.)		&=\dim (x\E t_i.)\\&\leq c_i-c_1\\
\implies\dim(\ker x\cap\E t_i.)	&=\dim(\ker(x|\E t_i.))	\\
								&\geq t_i-(c_i-c_1)\\
								&=r_i+c_1\geq r_i.
\end{align*}
Let $k=d-\dim(\im x)$.
Substituting $i=l$ in \cref{work40} yields, 
\begin{align*}
	\dim(\nicefrac{\im x}{\E t_j.})	&\leq r_{l-1}-r_j\leq d-r_j\\
	\implies\dim(\im x+\E t_j.)		&\leq(d-r_j)+t_j\\
	\implies\dim(\im x\cap\E t_j.)	&=\dim(\im x)+\dim\E t_j.-\dim(\im x+\E t_j.)\\
									&\geq(d-k)+t_j-(d-r_j+t_j)=r_j-k.
\end{align*}
Observe that since $x^2=0$, we have $\im x\subset\ker x$.
It now follows from \cref{strong,sch:cont} that there exists $V\in X_w$ such that,\[
	\im x\subset V\subset\ker x,
\]
i.e., $(V,x)\in T^*_XX_w$. 
Consequently, we have $x\in\mu(T^*_XX_w)$. 
\end{proof}

\subsection{Orbital Varieties in Type A}\label{Aorbital}
When $G=SL_n$, the $G$-orbits in $\Ni$ are indexed by the partitions of $n$. 
For $\boldsymbol\lambda$ a partition of $n$, the irreducible components of $\Ni[\lambda]\cap\u_B$ are indexed by the standard Young tableaux of shape $\boldsymbol\lambda$.
For $\mathsf T$ a standard Young tableau, we denote the corresponding orbital variety by $\mathcal O_\mathsf T$.

In this case, the relationship between conormal varieties of Schubert varieties, and orbital varieties, as described in \cref{orbitalConormal}, provides a geometric realization of the Robinson-Schensted correspondence.

\begin{proposition}[cf. \cite{MR929778}]
Suppose $\mathsf T$ is the left Robinson-Schensted tableau of some $w\in S_n$.
Then $\mathcal O_\mathsf T=\mu^B(T_{X^B}^*X^B_w)$.
\end{proposition}

\begin{proposition}
Let $\mathsf T$ be a standard Young tableau with exactly two columns. 
Then there exists a standard parabolic subgroup $P\subset G$, and a Schubert variety $X_w$ in $X=\nicefrac GP$, such that $\mathcal O_{\mathsf T}=\mu(T^*_XX_w)$.
\end{proposition}
\begin{proof}
Let $k$ be the number of boxes in the first column of $\mathsf T$, and let $P$ be the standard parabolic subgroup of $G$ corresponding to $\mathsf A_{n-1}\backslash\{\alpha_k\}$. 
The longest element $w_P$ of $W_P$ is given by\[
	w_P(i)=\begin{cases}k+1-i	&\text{for }i\leq k,\\
						n+1-k	&\text{for }i>k.
	\end{cases}
\]

Let $a_1,\cdots,a_k$ be the entries in the first column of $\mathsf T$, written in increasing order, i.e., top to bottom; and $b_1,\cdots,b_{n-k}$ the entries in the second column, also written in increasing order.
We consider the element $w\in S_n$ given by,\[
	w(i)=\begin{cases}a_i	&\text{for }i\leq k,\\
					b_{i-k}	&\text{for }i>k.
	\end{cases}
\]
Let $v=ww_P$.
Since $w\in S_n^P$, the Schubert variety $X^B_v$ is a fibre bundle over $X_w$ with fibre $\nicefrac BP$, and we have a Cartesian diagram,\[
\begin{tikzcd}
	X_{v}^B\arrow[d,"\operatorname{pr}"]\arrow[r,hook]	&X^B\arrow[d]	\\
	X^P_w\arrow[r,hook]					&X^P.
\end{tikzcd}
\]
The map $X^B\rightarrow X^P$ is precisely the quotient map $\nicefrac GB\rightarrow\nicefrac GP$. 
Consequently, $T^*_{X^B}X^B_{v}$ is a $\nicefrac BP$-fibre bundle over $\con$, with the map $T^*_{X^B}X^B_{v}\rightarrow \con$ being simply the restriction 
(to $T^*_{X^B}X^B_{v}\subset T^*X\subset X^B\times\Ni$)
of the map,
$$\operatorname{pr}\times\operatorname{id}_{\Ni}:X^B\times\Ni\rightarrow X\times\Ni.$$
This yields us $\mu^B(T^*_{X^B}X^B_w)=\mu(\con)$.
Finally, we verify that $\mathsf T$ is the left RSK-tableau of $v$, thus obtaining $\mathcal O_{\mathsf T}=\mu^B(T^*_{X^B}X^B_w)=\mu(\con)$.
\end{proof}

\begin{corollary}\label{recover}
Let $\mathsf T$ be a two-column standard Young tableau. 
Consider integers $0\leq j<i\leq n$, and the skew-tableau $\mathsf T\backslash\{1,\cdots,j,i+1,\cdots,n\}$.
Let $\mathsf T_i^j$ denote the tableau obtained from this skew-tableau
via `jeu de taquin'.
Then,\[
	\mathcal O_{\mathsf T}=\set{x\in\Ni}{J(x_i^j)\preceq\mathsf T_i^j},
\]
where $x_i^j$ is the square sub-matrix of $x$ with corners $(t_j+1,t_j+1)$ and $(t_i,t_i)$, $J(x_i^j)$ denotes the Jordan type of $x_i^j$, and $\preceq$ denotes the inclusion order on the set of $G$-orbits $\Ni[\lambda]$.
\end{corollary}
\begin{proof}
This statement is proved in \cite{MR2134184}.
We explain here how it also follows as a consequence of \cref{mainOrbital,Aorbital}.

Since $x^2=0$, we have $(x_i^j)^2=0$ for all $i,j$.
Consequently, the inequality $J(x_i^j)\preceq T_i^j$ is equivalent to the inequality $rk(x_i^j)\leq f_i^j$, where $f_i^j$ is the number of boxes in the second column of $\mathsf T_i^j$.
On the other hand, it follows from \cref{mainOrbital,Aorbital} that\[
	\mathcal O_{\mathsf T}=\set{x\in\mathcal N}{\rk(x_i^j)\leq g_i^j},
\]
for certain integers $g_i^j$.
It is a simple exercise to verify that the integers $f_i^j$ and $g_i^j$ defined here are equal.
\end{proof}

%% file: remarks.tex
\section{A Type Independent Conjecture}
\label{sec:conj}
In this section, we assume that $X$ is a cominuscule Grassmannian corresponding to some Dynkin diagram.
We conjecture, for any Schubert variety $X_w\subset X$, the following equality,
\begin{align}\label{genEqn}
	\con =\mu^{-1}(\mu(\con ))\cap\pi^{-1}(X_w).
\end{align}
The question is well-posed in both set-theoretic and scheme-theoretic settings.

Suppose $X_w\subset X$ is a smooth Schubert subvariety.
We prove in \cref{whenSmooth} that \cref{genEqn} holds set-theoretically in this case.

Next, let $w_0$ denote the longest element in the Weyl group $W$.
We show in \cref{oppSmooth} that if $X_w\subset X$ is a Schubert variety such that the \emph{opposite} Schubert variety $X_{w_0w}$ is smooth, then \cref{genEqn} holds scheme-theoretically. 
This is a straightforward corollary to \cite[Theorem 1.1]{conormal2017}.

When $X$ is the usual Grassmannian or the symplectic Grassmannian, the set-theoretic version is a consequence of \cref{finalEquations,mainOrbital}.
In type  B, the only cominuscule Grassmannian is the one corresponding to $G=SO_{2n+1}$, and the cominuscule root $\alpha_1$.
In this cases, one easily verifies that for each $w\in W^P$, either $X_w$ is smooth, or $X_{w_0w}$ is smooth, hence settling the set-theoretic version of our conjecture for all cominuscule Grassmannians in types A, B, and C.

One would like to know in which of these cases \cref{genEqn} holds scheme-theoretically, and also whether \cref{genEqn} holds for types D and E.
If it does, can we find a uniform, type independent proof?

\begin{proposition}\label{whenSmooth}
Suppose $X_w$ is smooth.
Then the conormal variety $\con $ satisfies \cref{genEqn} set-theoretically.
\end{proposition}
\begin{proof}
A Schubert variety $X_w$ in a cominuscule Grassmannian $X$ is smooth if and only if $X_w$ is homogeneous for some standard parabolic subgroup $L$, see \cite{billey2010smooth}.

Suppose $X_w$ is homogeneous for some standard parabolic subgroup $L$; let $S_L$ be the corresponding subset of $S$, and $w_L$ the longest word of $W$ supported on $S_L$.
Then $w$ is the minimal representative of $w_L$ in $W^P$. 
Further, the subspace $\u_w\subset\u$ from \cref{defn:uw} is precisely,\[
	\u_w=\bigoplus\limits_{\alpha\geq\gamma,\,Supp(\alpha)\not\in S_L}\lgl_\alpha
\]
In particular, $\u_w$ is $L$-stable.

The quotient map $\nicefrac GB\rightarrow\nicefrac GP$ induces an isomorphism $\nicefrac LB\xrightarrow\sim X_w$, and the conormal variety $\con \rightarrow X_w$ is simply the vector bundle $L\times^B\u_w\rightarrow\nicefrac LB$.
Consequently, we have,\begin{align}\label{smoothOrbital}
	\mu(\con )=\set{Ad(l_0)x_0}{l_0\in L,\,x_0\in\u_w}.
\end{align}

Now, consider some $(l,x)\in G\times^P\u$, satisfying $\pi(l)\in X_w$ and $\mu(l,x)\in\mu(\con )$.
We may assume, without loss of generality, that $l\in L$.
As a consequence of \cref{smoothOrbital}, there exist $l_0\in L$, and $x_0\in\u_w$, such that,
\begin{align*}
	\mu(l,x)	&=Ad(l)x=Ad(l_0)x_0\\
	\implies x	&=Ad(l^{-1}l_0)x_0.
\end{align*}
Now, since $\u_w$ is $L$-stable, we have $x\in\u_w$, hence $(l,x)\in \con $.
\end{proof}

\begin{proposition}\label{oppSmooth}
Suppose the opposite Schubert variety $X_{w_0w}$ is smooth for some $w\in W^P$.
Then $\con $ satisfies \cref{genEqn} scheme-theoretically.
\end{proposition}
\begin{proof}
Let $\dynk_0$ denote the Dynkin diagram of $G$, and let $\dynk$ be the corresponding extended Dynkin diagram.
The loop group $LG=G(\k[t,t^{-1}])$ is an affine Kac-Moody group corresponding to the extended Dynkin diagram $\dynk$.
Let $\mathcal G_0$, $\mathcal G_d$, and $\mathcal P$ be parabolic subgroups of $LG$ corresponding to the subsets $\dynk\backslash\{\alpha_0\},\dynk\backslash\{\alpha_d\}$, and $\dynk\backslash\{\alpha_0,\alpha_d\}$ respectively.

Following \cite{conormal2017}, there exists an embedding $\phi:\con \rightarrow\nicefrac{LG}{\mathcal P}$ such that $\phi(\con )$ is an open subset of some Schubert subvariety of $\nicefrac{LG}{\mathcal P}$.
Further, we can identify the structure map $\pi$ and the Springer map $\mu$ as the restriction to $\phi(\con )$ of the quotient maps $\pi_d:\nicefrac{LG}{\mathcal P}\rightarrow\nicefrac{LG}{\mathcal G_d}$ and $\pi_0:\nicefrac{LG}{\mathcal P}\rightarrow\nicefrac{LG}{\mathcal G_0}$ respectively.

Now, for any Schubert variety $Y\subset\nicefrac{LG}{\mathcal P}$, we have the scheme-theoretic equality,\begin{align*}
	Y=\pi_0^{-1}(\pi_0(Y))\cap\pi_d^{-1}(\pi_d(Y)).
\end{align*}
From this, we deduce that \cref{genEqn} holds for $\con $ scheme-theoretically.
\end{proof}